\documentclass[11pt]{amsart}
\usepackage[utf8]{inputenc}
\usepackage[english]{babel}

\usepackage{cite}
\usepackage{amscd,amssymb}
\usepackage{amsmath}
\usepackage{amsthm}
\usepackage{mathrsfs}
\usepackage{graphicx}
\usepackage{fancybox}
\usepackage{epic,eepic}
\usepackage{amstext}
\usepackage{mathtools}
\usepackage{esint}
\usepackage[square,numbers]{natbib}
\usepackage{booktabs}
\usepackage{hyperref}
\usepackage{comment}
\usepackage{mathtools}
\usepackage{tikz,pgfplots}
\usepackage{subcaption}
\usepackage{cancel}
\usepackage{bm}
\usepackage{tikz}
\usepackage{tikz-3dplot}
\usepackage{tikz-cd}
\usepackage[most]{tcolorbox}
\usepackage{cleveref}
\newcounter{quotecount}
\newcommand{\newquote}[1]{\vspace{0.5cm}\refstepcounter{quotecount}%
     \parbox{10cm}{\em #1}\hspace*{2cm}(\arabic{quotecount})\\[0.5cm]}

\def\H{\mathcal{H}}

\definecolor{myBlue}{RGB}{113,104,238} %
\definecolor{myGreen}{RGB}{114,175,30}
\definecolor{myRed}{RGB}{180,50,50}
\definecolor{myOrange}{RGB}{225,92,22}

\newtheorem{theorem}{Theorem}[section]
\newtheorem{corollary}[theorem]{Corollary}
\newtheorem{lemma}[theorem]{Lemma}
\newtheorem{assumption}[theorem]{Assumption}
\newtheorem{example}[theorem]{Example}

\theoremstyle{definition}
\newtheorem{definition}[theorem]{Definition}
\theoremstyle{remark}
\newtheorem{remark}[theorem]{Remark}
\numberwithin{theorem}{section}
\numberwithin{equation}{section}
\numberwithin{table}{section}
\numberwithin{figure}{section}

\allowdisplaybreaks

\usepackage{xcolor}

\newcommand{\0}{\mathaccent23}

\newcommand\grad{\operatorname{grad}}

\renewcommand\div{\operatorname{div}}

\newcommand\curl{\operatorname{curl}}

\usetikzlibrary{quotes}
\usetikzlibrary{decorations.markings,arrows}
\usepackage{soul}

\usepackage{diagbox}

\let \vec \bm
\hypersetup{hidelinks}
\begin{document}

\title[]{Finite element exterior calculus for spectra and pseudospectra of advection-diffusion of differential forms}
\author{Daniele Boffi}
\address{Computer, Electrical and Mathematical Sciences and Engineering,
King Abdullah University of Science and Technology}
\address{Dipartimento di Matematica ``F. Casorati'',
University of Pavia}
\address{IMATI ``E. Magenes'', CNR, Pavia}
\email{daniele.boffi@kaust.edu.sa}

\author{Kaibo Hu}
\address{Mathematical Institute, University of Oxford,
Andrew Wiles Building, Woodstock Road,
Oxford OX2 6GG, UK}
\email{kaibo.hu@maths.ox.ac.uk}

\author{Yizhou Liang}
\address{Mathematical Institute, University of Oxford,
Andrew Wiles Building, Woodstock Road,
Oxford OX2 6GG, UK}
\email{yizhou.liang@maths.ox.ac.uk}

\author{Umberto Zerbinati}
\address{Mathematical Institute, University of Oxford,
Andrew Wiles Building, Woodstock Road,
Oxford OX2 6GG, UK}
\email{umberto.zerbinati@oriel.ox.ac.uk}

\dedicatory{Dedicated to Professor Douglas N. Arnold\\ for his pioneering contributions to finite element exterior calculus}

\maketitle

\begin{abstract}
Numerical investigations of dynamo action have remained active in fluid mechanics over the past decades, presenting numerous challenges and open problems. Meanwhile, the development of structure-preserving methods and finite element exterior calculus (FEEC) inspires a revisit of numerical dynamo studies and an exploration of existing open questions in computation. In this paper, we present a FEEC approach for dynamo problems. In particular, we investigate structure-preserving finite element schemes for computing the spectra and pseudospectra of advection-diffusion operators of differential forms. The schemes and their analysis are based on finite element de~Rham complexes.
\end{abstract}

\section{Introduction}
Dynamo theory is concerned with the physical mechanisms responsible for the generation of magnetic fields in astrophysical bodies and electrically conducting fluids. Within this framework, kinematic dynamo theory specifically investigates the types of fluid motions capable of producing exponential amplification of a magnetic field when magnetic diffusivity is small. Given its evident importance for solar physics and geophysics, dynamo theory is a key branch of \textit{topological (magneto)hydrodynamics}~\cite{arnoldkhesin2021,ricca1996topological}.

Mathematically, a simplified problem (linear or kinematic dynamo theory) involves solving for the magnetic field $\bm{B}$ via the magnetic advection-diffusion equation
\begin{equation}\label{eq:MHD}
\partial_t \bm{B} - \nabla \times (\bm{u} \times \bm{B}) + R_m^{-1} \nabla\times \nabla\times  \bm{B} = 0,
\end{equation}
with a fixed velocity field $\bm{u}$ and magnetic Reynolds number $R_m$.
In this simplified model, one ignores the feedback of the magnetic field on the fluid. This simplification is justified when the magnetic field is small, i.e., at the onset of dynamo action.

The following problem was formulated (see Problem 1.2, §5 in \cite{arnoldkhesin2021}, with references to  Ya.B. Zeldovich and A.D. Sakharov):
\begin{quote}
Does there exist a divergence-free velocity field $\bm{u}$ in a domain $M$ such that the energy $\mathcal{E}(t) := \|\bm{B}(t)\|_{L^{2}(M)}^{2}$ of the magnetic field $\bm{B}(t)$ grows exponentially in time for some initial field $\bm{B}(0) = \bm{B}_{0}$ and for arbitrarily low diffusivity?
\end{quote}

Numerical investigations have drawn attention from both fluid mechanists and mathematicians throughout the development of dynamo theory \cite{galloway1984numerical,archontis2003numerical,bouya2013revisiting,arnold1983growth}. A major computational challenge is to find the eigenvalue with the largest real part of the advection-diffusion operator (with a negative sign)
\begin{equation}
\label{eq:advectionDiffusionOperator}
\mathcal{L}_{\bm{u}} \colon \bm{B} \mapsto \nabla \times (\bm{u} \times \bm{B}) - R_m^{-1} \nabla \times \nabla \times \bm{B}.
\end{equation}
The leading eigenvalue corresponds to the fastest-growing mode.

In 1983, Arnold and Korkina \cite{arnold1983growth} conducted numerical experiments using the ABC flow for $\bm{u}$, a flow exhibiting chaotic behavior and a candidate for dynamo action. Their computations were based on Galerkin methods with magnetic Reynolds number $R_m \leq 19$ due to computational limitations. This computation was followed by several works, see, e.g., \cite{galloway1984numerical,archontis2003numerical,bouya2013revisiting}.

Arnold and Khesin~\cite[Page~305]{arnoldkhesin2021} further commented that

\newquote{\label{arnold-khesin}
\dots It is still unknown whether this field is a fast kinematic dynamo, e.g., whether an exponentially growing mode of $\bm{B}$ survives as $R_m \to \infty$.

\dots

Numerically, the kinematic fast dynamo problem is the first eigenvalue problem for matrices of the order of many million, even for reasonable Reynolds numbers (of the order of hundreds). The physically meaningful magnetic Reynolds numbers $R_m$ are of order of magnitude $10^{8}$. The corresponding matrices are (and will remain) beyond the reach of any computer.
}

In parallel, progress relevant to MHD problems has been made in the numerical analysis community over the past decades. Around the time of the Arnold--Korkina computations~\cite{arnold1983growth}, spurious solutions (particularly spurious eigenmodes) in computational electromagnetism arising from standard scalar (Lagrange) finite elements drew attention \cite{bossavit2002solving}. The Nédélec element \cite{1980Nedelec,1986Nedelec} led to discretizations based on edge degrees of freedom, mimicking physical structures as also achieved in staggered-grid approaches such as the Yee scheme. Later, Bossavit \cite{1998Bossavit} and Hiptmair \cite{2002Hiptmair} realized that the Nédélec element, together with other finite elements discovered independently, fits into a discrete de~Rham complex and is a special case of Whitney forms in geometric integration theory \cite{1957Whitney}. This perspective inspired the systematic development of finite element differential forms and finite element exterior calculus (FEEC) \cite{2006ArnoldFalkWinther,2010ArnoldFalkWinther}. The Nédélec element, originally motivated by the need to eliminate spurious eigenmodes, has become standard in computational electromagnetism \cite{2003Monk}.

The underlying finite element de~Rham complex also enables structure preservation, which is particularly important in problems involving nonlinear coupling, such as MHD. FEEC-based approaches for MHD problems have been developed, with several key quantities preserved exactly at the discrete level, such as the magnetic Gauss law ($\nabla \cdot \bm{B} = 0$), energy, and various notions of helicity \cite{hiptmair2018fully,2017HuMaXu,2021HuLeeXu,laakmann2023structure,mao2025incompressibility,zhang2024decoupled,ma2016robust,he2025topology,gawlik2022finite}.

All the above approaches address MHD source problems. Meanwhile, mixed finite elements and FEEC for Maxwell eigenvalue problems have been systematically developed \cite{2006ArnoldFalkWinther,2010Boffi}. However, to the best of our knowledge, structure-preserving finite element computation and analysis for MHD eigenvalue problems, particularly for the dynamo problem, remain open.

The advancements in FEEC for MHD source problems and Maxwell eigenvalue problems motivate several questions about the Arnold--Korkina 1983 computations and subsequent efforts:
\begin{enumerate}
  \item Do these computations exhibit spurious eigenmodes, as no numerical analysis has been presented, to the best of our knowledge?
  \item Can we achieve computations for large and physically relevant Reynolds numbers?
\end{enumerate}

The aim of this paper is to investigate the computation of dynamo action using structure-preserving finite element methods. Specifically, we propose a finite element scheme for solving the MHD eigenvalue problem (the most natural choice from a FEEC perspective) and analyze its convergence. We provide numerical results and compare them with existing results from the fluid mechanics literature.

Moreover, advances in pseudospectra over the past decades (see \cite{2005Trefethen} for a summary) call into question the first claim in the quotation \eqref{arnold-khesin}, namely that ``\textit{numerically, the kinematic fast dynamo problem is the first eigenvalue problem}.'' Indeed, it is well known that for problems involving non-normal operators, eigenvalues can provide misleading information about instability, particularly in transient regimes and turbulent flows \cite{2005Trefethen,2007Schmid}. Additional insight, especially into transient dynamics, is captured by the pseudospectra. Consequently, computing pseudospectra and comparing the results with eigenvalue analyses are important topics. Although finite element methods for eigenvalue problems and their convergence properties have been extensively studied over the past decades \cite{2006ArnoldFalkWinther,2010Boffi}, pseudospectra in this context have remained largely unexplored. In this paper, we help bridge this gap by establishing convergence results for the finite element approximation of pseudospectra.

In fact, the linear dynamo operator (magnetic advection-diffusion) is a special case of advection-diffusion for differential forms. Depending on the degree of the differential forms, other examples include scalar advection-diffusion and the Fokker--Planck equation~\cite{arnoldkhesin2021}. In this paper, we work in the generality of advection-diffusion for differential forms, thus also addressing these applications.

The rest of this paper is organized as follows.
 Section \ref{sec:notation} introduces notation and preliminaries.  Section \ref{sec:pseudospectra} establishes results on the approximation of pseudospectra.  Sections \ref{sec:convergence1} and \ref{sec:convergence2} establish the operator convergence needed for verifying the approximation of spectra and pseudospectra for the half and full Hodge Laplacian, respectively. Section \ref{sec:numerical} presents numerical experiments. %
 Section \ref{sec:conclusions} provides concluding remarks.

\section{Notation and preliminaries}\label{sec:notation}

Let $\Omega\subset \mathbb R^n$ be a bounded Lipschitz domain with polyhedral boundary. In this paper, the domain $\Omega$ is allowed to have non-trivial topology. Let $\Lambda^k_{\mathbb C}(\Omega)$ denote the space of smooth complex-valued differential $k$-forms on $\Omega$. The smooth complex-valued de~Rham complex is given by
\begin{equation*}
    \begin{tikzcd}
        0\ar{r} & \Lambda^0_{\mathbb C}(\Omega)\ar{r}{d^0} &\Lambda^1_{\mathbb C}(\Omega)\ar{r}{d^1} & \cdots \ar{r}{d^{n-1}}& \Lambda^n_{\mathbb C}(\Omega)\ar{r}& 0,
    \end{tikzcd}
\end{equation*}
where the differentials $d^k$ are exterior derivatives satisfying $d^{k+1}\circ d^k=0$. In this paper, we assume that all complexes that appear are closed, in the sense that all the differential operators
 have closed range. Let $u$ be a smooth vector field. The \emph{Lie derivative} $L_{u}:\Lambda^{k}_{\mathbb C}(\Omega)\rightarrow\Lambda^k_{\mathbb C}(\Omega)$ on smooth differential $k$-forms is defined using the exterior derivative and the contraction operator with $ u$ as:
\begin{equation*}
    L_u \omega = \mathrm{i}^{k+1}_{ u}d^k \omega + d^{k-1}\mathrm{i}^{k}_{ u} \omega,\quad \forall  \omega\in \Lambda^k_{\mathbb C}(\Omega),
\end{equation*}
where $\mathrm{i}^{k}_{ u}: \Lambda^k_{\mathbb C}(\Omega)\rightarrow \Lambda^{k-1}_{\mathbb C}(\Omega)$ denotes the interior product (contraction) with $ u$. The $L^2$-inner product of differential $k$-forms on the domain $\Omega$ is defined as:
\begin{equation*}
    ( \omega,\eta) : =\int_{\Omega}\omega\wedge\star \bar\eta,
\end{equation*}
where $\star: \operatorname{Alt}^kV\rightarrow \operatorname{Alt}^{n-k}V$ is the Hodge operator satisfying $\star(\star \omega) = (-1)^{k(n-k)}\omega$ for any $\omega\in \Lambda^k_{\mathbb C}(\Omega)$. The completion of $\Lambda^k_{\mathbb C}(\Omega)$ in the corresponding norm $\|\cdot\|_{L^2}^2 = (\cdot,\cdot)$ defines the Hilbert space $L^2\Lambda^k(\Omega)$. We then define the Sobolev spaces $H^s\Lambda^k(\Omega)$ consisting of differential forms for which the coefficients belong to $H^s(\Omega)$ for any real number $s$. The norm is denoted by $\|\cdot\|_{H^s}$. Define
\begin{equation*}
    H\Lambda^k(\Omega):=\{ \omega\in L^2\Lambda^k(\Omega):d^k\omega\in L^2\Lambda^{k+1}(\Omega)\}.
  \end{equation*}
We obtain the $L^2$ de~Rham complex with the domain complex
\begin{equation}\label{eq:deRham}
    \begin{tikzcd}
        0\ar{r} & H\Lambda^0(\Omega)\ar{r}{d^0} &H\Lambda^1(\Omega)\ar{r}{d^1} & \cdots \ar{r}{d^{n-1}}& H\Lambda^n(\Omega)\ar{r}& 0.
    \end{tikzcd}
\end{equation}
We also use the de~Rham complex with homogeneous boundary conditions
\begin{equation*}
    \begin{tikzcd}
        0\ar{r} & \mathring{H}\Lambda^0(\Omega)\ar{r}{d^0} &\mathring{H}\Lambda^{1}(\Omega)\ar{r}{d^1} & \cdots \ar{r}{d^{n-1}}& \mathring{H}\Lambda^n(\Omega)\ar{r}& 0.
    \end{tikzcd}
\end{equation*}
with
\begin{equation*}
    \mathring{H}\Lambda^k(\Omega)=\{\omega\in H\Lambda^k(\Omega):\operatorname{Tr}(\omega) =0\}.
\end{equation*}
The $k$-th cohomology space of \eqref{eq:deRham} is isomorphic to the space $\mathcal{H}^k$ of harmonic $k$-forms on $\Omega$:
\begin{equation*}
    \mathcal{H}^k = \{\omega\in H\Lambda^k(\Omega):d^k\omega=0,(\omega,d^{k-1}\eta) = 0,\forall\eta\in H\Lambda^{k-1}(\Omega)\}.
\end{equation*}
In analogy with $H\Lambda^k(\Omega)$, we define the space
\begin{equation*}
    H^*\Lambda^k(\Omega):= \star H\Lambda^{n-k}(\Omega)=\{ \omega\in L^2\Lambda^k(\Omega):\delta^k\omega\in L^2\Lambda^{k-1}(\Omega)\} ,
\end{equation*}
where $\delta^{k}:\Lambda^k_{\mathbb C}(\Omega)\rightarrow \Lambda^{k-1}_{\mathbb C}(\Omega)$ is the coderivative operator
\begin{equation*}
    \star\delta^k=(-1)^kd^{n-k}\star.
\end{equation*}
For any $v\in \Lambda^k_{\mathbb C}(\Omega)$ and $w\in  \Lambda^{k+1}_{\mathbb C}(\Omega)$, Stokes' formula holds:
\begin{equation*}
    (d^kv,w) = (v,\delta^{k+1}w) + \int_{\partial\Omega}\operatorname{Tr}v\wedge\operatorname{Tr}(\star \bar{w}).
\end{equation*}
The norm of $H^*\Lambda^k(\Omega)$ is defined as $\|\omega\|^{2}_{H^*\Lambda} = \|\omega\|^{2}_{L^2} +\|\delta^k\omega\|^{2}_{L^2}$. The space $H^*\Lambda^k(\Omega)$ with homogeneous boundary conditions is defined as:
\begin{equation*}
    \mathring{H}^*\Lambda^k(\Omega):=\star \mathring{H}\Lambda^{n-k}(\Omega)=\{\omega\in H^*\Lambda^k(\Omega):\operatorname{Tr}(\star\omega) =0\} .
\end{equation*}
The adjoint of \eqref{eq:deRham} (more precisely, the domain of the adjoint of the $L^{2}$ de~Rham complex with domain \eqref{eq:deRham}) is
\begin{equation}\label{eq:deRham_dual}
    \begin{tikzcd}
        0\ar{r} & \mathring{H}^*\Lambda^n(\Omega)\ar{r}{\delta^n} &\mathring{H}^*\Lambda^{n-1}(\Omega)\ar{r}{\delta^{n-1}} & \cdots \ar{r}{\delta^{1}}& \mathring{H}^*\Lambda^0(\Omega)\ar{r}& 0.
    \end{tikzcd}
\end{equation}
This provides a further characterization of the harmonic forms:
\begin{equation*}
    \mathcal{H}^k =\{\omega\in H\Lambda^k(\Omega)\cap \mathring{H}^*\Lambda^k(\Omega):d^k\omega=0,\delta^k\omega=0\}.
\end{equation*}
For $0\leq s\leq 1$,  a domain $\Omega$ is called $s$-regular if for any $v\in H\Lambda^k(\Omega)\cap \mathring{H}^*\Lambda^k(\Omega)$, we have
\begin{equation}\label{eq:s-regular}
    \|v\|_{H^s}\leq C(\|v\|_{L^2} + \|d^kv\|_{L^2} + \|\delta^kv\|_{L^2})
\end{equation}
with a constant $C>0$ independent of $v$. A Lipschitz domain is $\frac{1}{2}$-regular \cite{2001Mitrea}, and for a Lipschitz polyhedron in $\mathbb R^3$, we can choose $1/2<s\leq 1$ \cite{1998Amrouche}. A convex domain is $1$-regular.  In this paper, we assume that the domain $\Omega$ is $s$-regular with $1/2< s\leq 1$.

Let $\{\mathcal{T}_h\}$ be a family of shape-regular simplicial triangulations of $\Omega$ with the mesh size $h:=\max_{K\in\mathcal{T}_h}\operatorname{diam}(K)$. In this paper, the numerical schemes are based on the discretization of the dual complex $\eqref{eq:deRham_dual}$. Hence, we consider the finite element spaces $\Lambda_h^k\subset \mathring{H}^*\Lambda^k(\Omega)$ and the corresponding complex
\begin{equation}\label{discrete-delta-complex}
    \begin{tikzcd}
        0\ar{r} &\Lambda_h^n\ar{r}{\delta^n} &\Lambda_h^{n-1}\ar{r}{\delta^{n-1}} & \cdots \ar{r}{\delta^{1}}& \Lambda_h^0\ar{r}& 0.
    \end{tikzcd}
\end{equation}
\begin{remark}
The reason we choose to base all subsequent developments on \eqref{eq:deRham_dual} and \eqref{discrete-delta-complex}, rather than the more conventional \eqref{eq:deRham}, originates from the MHD problem \eqref{eq:MHD}. In that system, $\bm B$ is a 2-form satisfying $d\bm B = 0$ (i.e., $\div \bm B = 0$). The diffusion term in \eqref{eq:MHD} reduces to $d\delta \bm B$, which we discretize in weak form as $(\delta \bm B, \delta \bm C)$ using a test 2-form $\bm C$. This approach naturally requires a conforming discretization built around the dual de Rham complex \eqref{eq:deRham_dual}.

The preference for \eqref{eq:deRham_dual} over \eqref{eq:deRham} is largely notational rather than substantial. In three-dimensional vector proxy notation, the dual complex \eqref{eq:deRham_dual} corresponds to $-\nabla$, $\curl$, and $-\div$, which is essentially equivalent to the standard complex \eqref{eq:deRham}. In the broader framework of finite element exterior calculus (FEEC), conforming discretizations of \eqref{eq:deRham} are well established \cite{2006ArnoldFalkWinther,2010ArnoldFalkWinther}. One can then apply a Hodge star to obtain a conforming discretization of \eqref{eq:deRham_dual}. Similarly, composing bounded cochain projections of \eqref{eq:deRham} with a Hodge star yields the corresponding projections for \eqref{eq:deRham_dual}.
\end{remark}

One may derive constructions of $\Lambda_h^k$ from standard results in \cite{2006ArnoldFalkWinther,2010ArnoldFalkWinther}. For \eqref{eq:deRham}, we have conforming finite element spaces $\widetilde{\Lambda}^{k}_{h}\subset \mathring{H}\Lambda^k(\Omega)$. Define $\Lambda^{k}_{h}:=\star\widetilde{\Lambda}^{n-k}_{h}\subset \mathring{H}^*\Lambda^k$, which leads to a conforming finite element subcomplex of \eqref{eq:deRham_dual}.

In the rest of this paper, we will rely on the existence of bounded cochain projections.
\begin{assumption}\label{assump:interpolation}
We assume that there exist bounded projections $\tilde\pi_h^k: \0 H\Lambda^k(\Omega)\rightarrow \widetilde{\Lambda}^{k}_{h}$ such that
\begin{equation*}
    d^k\tilde\pi_h^{k} = \tilde\pi_h^{k+1
    }d^k,
\end{equation*}
i.e., the following diagram commutes:
\begin{equation*}
    \begin{tikzcd}
    0\ar[r,""]&\mathring{H}\Lambda^0(\Omega) \arrow[r, "d^0"] \arrow[d, "\tilde\pi_h^0"]   & \cdots\arrow[r, "d^{n-1}"] \arrow[d, ""]  &  \mathring{H}\Lambda^n(\Omega)\ar[r,""] \arrow[d, "\tilde\pi_h^n"] &   0\\
    0\ar[r,""]& \widetilde\Lambda_h^0\arrow[r, "d^0"] & \cdots\ar[r,"d^{n-1}"] &  \widetilde\Lambda_h^n \ar[r,""]&0.
\end{tikzcd}
\end{equation*}
Moreover, for any $u\in \mathring{H}\Lambda^k(\Omega)\cap H^t\Lambda^k(\Omega)$ with $1/2<t\leq 1$, we have
\begin{equation*}
    \|u -\tilde{\pi}_h^ku\|_{L^2}\leq
        Ch^t|u|_{H^t},
\end{equation*}
where the constant $C>0$ depends on $t$ but can grow unboundedly as $t\rightarrow 1/2$.
\end{assumption}

An example of interpolation operators satisfying the requirements in Assumption \ref{assump:interpolation} can be found in \cite[Sections~22, 23]{2021ErnGuermond}. With these operators, we can define  $\pi_h^k:\mathring{H}^*\Lambda^k(\Omega)\rightarrow \Lambda_h^k$ by
\begin{equation*}
    \pi_h^k = (-1)^{k(n-k)}\star\tilde{\pi}_h^{n-k}\star.
\end{equation*}

At the end of this section, we introduce the discrete harmonic forms $$\H_h^{k}:=\{  u\in  \Lambda_h^k:\delta^k u=0,( u,\delta^{k+1} v)=0, \forall v\in \Lambda^{k+1}_h\}.$$ The following holds  \cite{2006ArnoldFalkWinther}.
\begin{lemma}\label{lem:discrete_har:appro}
    For any discrete harmonic $k$-form $ q\in \H_{h}^k$, there exists $ r\in \mathcal H^k$ such that $\| r\|_{L^2}\leq \| q\|_{L^2}$ and
    \begin{equation*}
        \| q- r\|_{L^2}\leq \|(I-\pi_h^k) r\|_{L^2}\leq Ch^s\|r\|_{H^s}\leq Ch^s\|q\|_{L^2}.
    \end{equation*}
\end{lemma}

\section{Approximation of pseudospectra}\label{sec:pseudospectra}

The motivation of Arnold and Korkina~\cite{arnold1983growth} for studying the eigenvalue problem
\begin{equation}\label{eq:dyn:eig}
    \mathcal{L}_{u}\, B = \lambda\, B,
\end{equation}
where $\mathcal{L}_u$ is the magnetic advection-diffusion operator of \eqref{eq:advectionDiffusionOperator} or, more generally, the Lie-derivative advection-diffusion operator \eqref{k-form-eigenvalue} acting on $k$-forms, is its connection with the dynamo problem through standard results from semigroup theory. Indeed, given a bounded linear operator $\mathcal{L}:X\to X$ on a Banach space $X$ that generates a $C_0$ semigroup $\{e^{t\mathcal{L}}\}_{t\geq 0}$, one has~\cite[Theorem 15.3]{2005Trefethen}
\begin{equation}\label{eq:growth:bound}
    \limsup_{t\to\infty}\; t^{-1} \log \| e^{t\mathcal{L}}\|_{\mathcal{L}(X;X)} = \alpha(\mathcal{L}) := \sup \left\{ \Re(\lambda) : \lambda \in \sigma(\mathcal{L})\right\},
\end{equation}
where $\alpha(\mathcal{L})$ is the \emph{spectral abscissa} of $\mathcal{L}$ and $\sigma(\mathcal{L}):=\{\lambda\in\mathbb C : \lambda I - \mathcal L \text{ is not bijective}\}$ is its spectrum. Identity \eqref{eq:growth:bound} would justify reducing the question of exponential growth of magnetic energy $\|B(t)\|_{L^2}^2$ to the eigenvalue problem~\eqref{eq:dyn:eig}.
However, the magnetic advection-diffusion operator $\mathcal{L}_u : D(\mathcal{L}_u)\subset L^2\Lambda^k(\Omega)\to L^2\Lambda^k(\Omega)$ is unbounded, and~\eqref{eq:growth:bound} does not directly apply.  Furthermore, as observed by Trefethen and Embree \cite{2005Trefethen} and by Schmid \cite{2007Schmid}, eigenvalues are often misleading indicators of stability for non-normal operators, i.e., operators for which $\mathcal{L}\mathcal{L}^*\neq \mathcal{L}^*\mathcal{L}$, so that even when \eqref{eq:growth:bound} is available the spectral abscissa may overlook the relevant transient dynamics. The dynamo operator $\mathcal{L}_u$ is non-normal in general.

The standard remedy in both cases is to replace the spectrum by the \emph{pseudospectrum}.

\begin{definition}[$\varepsilon$-pseudospectrum]\label{def:pseudospectrum}
    Let $X$ be a complex Hilbert space and $\mathcal{A}:D(\mathcal{A})\subset X\to X$ a closed linear operator. For any $\varepsilon>0$, the $\varepsilon$-pseudospectrum of $\mathcal{A}$ is
    \begin{equation}\label{eq:pseudospectrum}
        \sigma_\varepsilon(\mathcal{A}) := \left\{\lambda\in \mathbb C : \|(\lambda I - \mathcal{A})^{-1}\|_{\mathcal{L}(X;X)} > \varepsilon^{-1}\right\},
    \end{equation}
    with the convention that $\|(\lambda I -\mathcal A)^{-1}\|_{\mathcal L(X;X)} = +\infty$ whenever $\lambda \in \sigma(\mathcal{A})$, so that $\sigma(\mathcal{A})\subset \sigma_\varepsilon(\mathcal{A})$ for every $\varepsilon>0$.
\end{definition}

The pseudospectrum encodes the sensitivity of the spectrum to perturbations of the operator: for a normal operator one has $\sigma_\varepsilon(\mathcal A) = \{\lambda \in \mathbb C : \operatorname{dist}(\lambda,\sigma(\mathcal A))<\varepsilon\}$, while for non-normal operators $\sigma_\varepsilon(\mathcal A)$ may be much larger than this neighborhood. The following theorem gives the relevant counterpart of \eqref{eq:growth:bound} when $\mathcal L$ is unbounded but closed.

\begin{theorem}[Lower bounds on $\|e^{t\mathcal{L}}\|$; {\cite[Theorem 15.4]{2005Trefethen}}]\label{thm:gearhart-pruss}
    Let $\mathcal{L}:D(\mathcal{L})\subset X\to X$ be a closed linear operator on a Hilbert space $X$ that generates a $C_0$ semigroup $\{e^{t\mathcal{L}}\}_{t\geq 0}$. Then, for every $\varepsilon>0$,
    \begin{equation}\label{eq:gearhart-pruss}
        \varepsilon\sup_{t\geq 0} \|e^{t\mathcal{L}}\|_{\mathcal{L}(X;X)} \;\geq\; \alpha_\varepsilon(\mathcal{L}) := \sup\left\{ \Re(\lambda) : \lambda \in \sigma_\varepsilon(\mathcal{L})\right\},
    \end{equation}
    where $\alpha_\varepsilon(\mathcal{L})$ is the \emph{$\varepsilon$-pseudospectral abscissa} of $\mathcal{L}$.
\end{theorem}

The magnetic advection-diffusion operator $\mathcal{L}_u$ from \eqref{eq:advectionDiffusionOperator}, when interpreted distributionally as an unbounded operator $\mathcal{L}_u:D(\mathcal{L}_u)\subset L^2\Lambda^k(\Omega)\to L^2\Lambda^k(\Omega)$ with $u$ divergence-free, is closed on the Hilbert space $L^2\Lambda^k(\Omega)$ and generates a $C_0$ semigroup, and \Cref{thm:gearhart-pruss} applies. In view of \eqref{eq:gearhart-pruss}, the question of Arnold and Khesin recalled in the quotation (\ref{arnold-khesin}) admits the following pseudospectral reformulation: a divergence-free velocity field $u$ is a \emph{fast kinematic dynamo} if there exists $\lambda_0>0$ such that
\begin{equation}\label{eq:fast-dynamo}
    \alpha_\varepsilon(\mathcal{L}_u)\;\geq\; \lambda_0
\end{equation}
for all sufficiently large $R_m$ and sufficiently small $\varepsilon>0$.

The pseudospectrum is also the right object for the description of the dynamo's \emph{transient phase}, in which nonlinear effects can become dominant before the asymptotic regime described by eigenvalues is reached. Indeed, for any $\tau>0$ and any $\lambda \in \mathbb{C}$ with $\Re(\lambda)>0$, define $K\coloneqq \Re(\lambda)\| (\mathcal{L}-\lambda I)^{-1}\|$. If $K>1$, one has \cite[Theorem 15.4, equation 15.11]{2005Trefethen}
\begin{equation}\label{eq:transient:bound}
	\sup_{0\leq t\leq \tau} \|e^{t\mathcal{L}}\|_{\mathcal{L}(X;X)} \;\geq\; e^{\Re(\lambda)\tau}%
	\left(1+\frac{e^{\Re(\lambda)\tau}-1}{K}\right)^{-1}.
\end{equation}
Estimate \eqref{eq:transient:bound} can produce non-trivial growth bounds even when $\alpha(\mathcal L)=0$, a phenomenon entirely invisible to the spectrum.

For the dynamo problem this distinction is more than a technicality: knowing \emph{whether} the magnetic field grows is only half the question, since for any practical purpose one must also know \emph{on which time scale} the growth takes place. Identity \eqref{eq:growth:bound} captures only the limiting rate as $t\to\infty$ and is silent on this latter question. Two operators with the same spectrum, and hence the same eigenvalue-predicted growth rate, may produce semigroups whose magnetic energy reaches a prescribed amplification on radically different time scales. The pseudospectral abscissa $\alpha_\varepsilon(\mathcal L)$, by contrast, encodes exactly this information through \eqref{eq:gearhart-pruss} and \eqref{eq:transient:bound}. The following example, reduced to its bare $2\times 2$ matrix essence, makes the discrepancy explicit.

\begin{example}[Eigenvalues miss the transient]\label{ex:transient}
The simplest illustration of the gap between $\alpha(\mathcal{L})$ and $\alpha_\varepsilon(\mathcal{L})$ is finite dimensional. Fix $\lambda>0$ and $N>0$ and consider on $X=\mathbb C^2$ the two operators
\begin{equation*}
    A \;=\; \begin{pmatrix} \lambda & N \\ 0 & \lambda \end{pmatrix}, \qquad B \;=\; \begin{pmatrix} \lambda & 0 \\ 0 & \lambda \end{pmatrix},
\end{equation*}
so that $\sigma(A)=\sigma(B)=\{\lambda\}$ and the spectral abscissa $\alpha(A)=\alpha(B)=\lambda$ is identical: according to \eqref{eq:growth:bound}, both semigroups should display the same asymptotic exponential factor $e^{\lambda t}$. A direct computation gives
\begin{equation*}
    e^{tA} \;=\; e^{\lambda t}\begin{pmatrix} 1 & Nt \\ 0 & 1 \end{pmatrix}, \qquad e^{tB} \;=\; e^{\lambda t}\,I,
\end{equation*}
so that $\|e^{tB}\|_{\mathcal L(\mathbb C^2,\mathbb C^2)} = e^{\lambda t}$ realizes the eigenvalue prediction exactly, while
\begin{equation*}
    \|e^{tA}\|_{\mathcal L(\mathbb C^2,\mathbb C^2)} \;=\; e^{\lambda t}\,\frac{\sqrt{(Nt)^2+4}+Nt}{2} \;\sim\; Nt\,e^{\lambda t} \qquad \text{as $Nt\to\infty$,}
\end{equation*}
so that the prefactor multiplying the common asymptotic factor $e^{\lambda t}$ can, on any finite time scale, be made arbitrarily large by increasing $N$. With $\lambda=10^{-1}$ and $N=10^{2}$ at $t=1$ already $\|e^{A}\|\approx 110$ while $\|e^{B}\|\approx 1.1$: the spectrum predicts the same asymptotic growth for both operators, but the semigroup generated by $A$ has already produced two orders of magnitude of amplification before $B$ has even doubled. The two semigroups do reach a common amplification, but on radically different time scales: on the dynamo-relevant time scale $B$ grows so slowly that the eigenvalue-driven exponential is essentially invisible, whereas $A$ produces order-of-magnitude effective growth essentially from $t=0$.

Direct computation of the resolvent yields the asymptotic relations
\begin{equation*}
    \alpha_\varepsilon(B) \;=\; \lambda + \varepsilon, \qquad \alpha_\varepsilon(A) \;=\; \lambda + \sqrt{N\varepsilon} + o(\sqrt{\varepsilon}) \qquad \text{as $\varepsilon\to 0^+$.}
\end{equation*}
For the values above and $\varepsilon=10^{-2}$ this gives $\alpha_\varepsilon(B)\approx 0.11$ but $\alpha_\varepsilon(A)\approx 1.1$, an order-of-magnitude difference that \Cref{thm:gearhart-pruss} and the transient bound \eqref{eq:transient:bound} translate into a non-trivial lower bound on $\sup_{0\leq t\leq \tau}\|e^{tA}\|_{\mathcal L(\mathbb C^2,\mathbb C^2)}$ vastly exceeding the eigenvalue prediction $e^{\lambda\tau}$, while leaving the slow growth of $\|e^{tB}\|$ untouched. The pseudospectral abscissa thus separates, on the dynamo-relevant transient time scale, the two cases that the spectral abscissa cannot.
\end{example}

\medskip

The previous bounds make the $\varepsilon$-pseudospectrum the natural quantity to compute, but they only become useful if pseudospectra can themselves be approximated reliably by finite element methods. The main result of this section, which we use repeatedly in Sections \ref{sec:convergence1} and \ref{sec:convergence2}, asserts that any norm-convergent sequence of finite-rank operators yields a convergent approximation of the pseudospectrum of a compact operator.

\begin{lemma}[Resolvent factorization]\label{lem:resolvent-factorization}
    Let $X$ be a complex Banach space, let $S,T\in\mathcal{L}(X;X)$, and let $\lambda\in\mathbb C$ be such that $(\lambda I-S)^{-1} \in \mathcal L(X;X)$. Then
    \begin{equation}\label{eq:resolvent-factorization}
        \lambda I - T \;=\; \bigl(I - (T-S)(\lambda I - S)^{-1}\bigr)\,(\lambda I - S).
    \end{equation}
\end{lemma}
\begin{proof}
    Distributing the right-hand side of \eqref{eq:resolvent-factorization} and using $(\lambda I-S)^{-1}(\lambda I-S)=I$ gives $(\lambda I - S) - (T-S) = \lambda I - T$.
\end{proof}

\begin{theorem}[Quantitative perturbation of pseudospectra]\label{thm:pseudospectra:perturb}
    Let $X$ be a complex Hilbert space and let $S,T\in\mathcal{L}(X;X)$. Then, for every $\varepsilon>0$,
    \begin{equation}\label{eq:pseudospectra:perturb}
        \sigma_\varepsilon(T) \;\subseteq\; \sigma_{\varepsilon+\|T-S\|_{\mathcal L(X;X)}}(S) \qquad\text{and}\qquad \sigma_\varepsilon(S) \;\subseteq\; \sigma_{\varepsilon+\|T-S\|_{\mathcal L(X;X)}}(T).
    \end{equation}
\end{theorem}

\begin{proof}
    By symmetry between $S$ and $T$, it suffices to establish the first inclusion. Set $\eta:=\|T-S\|_{\mathcal L(X;X)}$ and let $\lambda\in\sigma_\varepsilon(T)$. If $\lambda\in\sigma(S)$, then $\lambda\in\sigma_{\varepsilon+\eta}(S)$ by Definition~\ref{def:pseudospectrum}. Otherwise the resolvent $(\lambda I-S)^{-1}$ is bounded, and we suppose, for the sake of contradiction, that $\lambda\notin \sigma_{\varepsilon+\eta}(S)$, i.e., \ $\|(\lambda I-S)^{-1}\|_{\mathcal L(X;X)} \leq (\varepsilon+\eta)^{-1}$. Notice that
    \begin{equation*}
        \|(T-S)(\lambda I - S)^{-1}\|_{\mathcal L(X;X)} \;\leq\; \eta(\varepsilon+\eta)^{-1} \;<\; 1,
    \end{equation*}
    so that, by a Neumann-series argument,\footnote{Let $M\in\mathcal L(X;X)$ with $q:=\|M\|_{\mathcal L(X;X)} < 1$. The geometric series $\sum_{k\geq 0} M^k$ converges absolutely in $\mathcal L(X;X)$, since by submultiplicativity $\sum_{k\geq 0}\|M^k\|_{\mathcal L(X;X)} \leq \sum_{k\geq 0} q^k = (1-q)^{-1}$. Its sum is a two-sided inverse of $I-M$: the partial sums satisfy $(I-M)\sum_{k=0}^n M^k = I-M^{n+1}$, and since $\|M^{n+1}\|_{\mathcal L(X;X)}\leq q^{n+1}\to 0$ one obtains $(I-M)^{-1}=\sum_{k\geq 0} M^k$ with $\|(I-M)^{-1}\|_{\mathcal L(X;X)} \leq (1-q)^{-1}$.} the operator $I-(T-S)(\lambda I - S)^{-1}$ is invertible in $\mathcal L(X;X)$ with
    \begin{equation*}
        \|(I-(T-S)(\lambda I - S)^{-1})^{-1}\|_{\mathcal L(X;X)} \;\leq\; \frac{1}{1-\eta(\varepsilon+\eta)^{-1}} \;=\; \frac{\varepsilon+\eta}{\varepsilon}.
    \end{equation*}
    The resolvent factorization \eqref{eq:resolvent-factorization} of Lemma~\ref{lem:resolvent-factorization} then yields $(\lambda I-T)^{-1} = (\lambda I-S)^{-1}\bigl(I-(T-S)(\lambda I-S)^{-1}\bigr)^{-1}$, so that
    \begin{equation*}
        \|(\lambda I-T)^{-1}\|_{\mathcal L(X;X)} \;\leq\; \|(\lambda I-S)^{-1}\|_{\mathcal L(X;X)}\,\frac{\varepsilon+\eta}{\varepsilon} \;\leq\; \frac{1}{\varepsilon+\eta}\cdot\frac{\varepsilon+\eta}{\varepsilon} \;=\; \varepsilon^{-1},
    \end{equation*}
    which contradicts $\lambda\in\sigma_\varepsilon(T)$. Therefore $\lambda\in\sigma_{\varepsilon+\eta}(S)$ and the inclusion \eqref{eq:pseudospectra:perturb} holds.
\end{proof}

\Cref{thm:pseudospectra:perturb} is a quantitative, two-sided estimate obtained from Kato's first stability estimate adapted to the pseudospectral setting. For more details, see \cite[Theorem 52.4]{2005Trefethen}. Specializing $S$ to a discrete approximation $T_n$ of $T$ yields the qualitative convergence statement that we will use throughout.

\begin{corollary}[Approximation of pseudospectra]\label{cor:pseudospectra:approx}
    Let $T\in\mathcal L(X;X)$ and $\{T_n\}_{n\in\mathbb N}\subset \mathcal L(X;X)$ satisfy
    \begin{equation}\label{eq:uniform:conv}
        \lim_{n\to\infty}\|T_n-T\|_{\mathcal{L}(X;X)} = 0.
    \end{equation}
    Then, for every $0<\delta<\varepsilon$, there exists $N\in\mathbb N$ such that
    \begin{equation}\label{eq:pseudospectra:incl}
        \sigma_\delta(T_n) \;\subseteq\; \sigma_\varepsilon(T) \qquad\text{and}\qquad \sigma_\delta(T) \;\subseteq\; \sigma_\varepsilon(T_n) \qquad\forall\, n\geq N.
    \end{equation}
\end{corollary}
\begin{proof}
    By \eqref{eq:uniform:conv} there exists $N\in\mathbb N$ such that $\|T_n-T\|_{\mathcal L(X;X)} < \varepsilon-\delta$ for every $n\geq N$. Applying \Cref{thm:pseudospectra:perturb} with $S=T_n$ and pseudospectral parameter $\delta$ yields
    \begin{equation*}
        \sigma_\delta(T_n) \;\subseteq\; \sigma_{\delta + \|T_n - T\|_{\mathcal L(X;X)}}(T) \;\subseteq\; \sigma_\varepsilon(T),
    \end{equation*}
    and the second inclusion follows analogously by exchanging the roles of $T$ and $T_n$.
\end{proof}

The verification of the uniform-convergence hypothesis \eqref{eq:uniform:conv} for the FEEC discretizations of the half and full Hodge--Laplace advection-diffusion eigenvalue problems is the main task of \Cref{sec:convergence1,sec:convergence2}. In particular, the convergence-rate estimates derived below also yield convergence-rate estimates for the pseudospectra via the above Corollary.

\begin{corollary}[Spectral convergence]\label{cor:spectral:conv}
    Let $T:X\to X$ be a compact linear operator and let $\{T_n\}_{n\in\mathbb N}\subset\mathcal L(X;X)$ satisfy \eqref{eq:uniform:conv}. Then, for every open set $U\subset\mathbb C$ containing $\sigma(T)$, there exists $N\in \mathbb N$ such that $\sigma(T_n)\subseteq U$ for every $n\geq N$.
\end{corollary}
\begin{proof}
	Since $T$ is compact, $\sigma(T)$ has at most $0$ as an accumulation point and is therefore a compact set. Because the set $\sigma(T)$ is compact and $U$ is an open neighborhood of $\sigma(T)$, there exists $\varepsilon>0$ such that $\sigma_\varepsilon(T)\subseteq U$. Applying Corollary~\ref{cor:pseudospectra:approx} with any $\delta\in(0,\varepsilon)$ yields $\sigma(T_n)\subseteq \sigma_\delta(T_n)\subseteq \sigma_\varepsilon(T)\subseteq U$ for all $n$ large enough, i.e., as $n\rightarrow \infty$.
\end{proof}

Corollary~\ref{cor:spectral:conv} recovers the classical Osborn-type spectral convergence result for non-self-adjoint compact eigenvalue problems~\cite{1991BabuskaOsborn}. In particular, the pseudospectral convergence of Corollary~\ref{cor:pseudospectra:approx} ensures that the discretizations analyzed in \Cref{sec:convergence1,sec:convergence2} produce neither spurious nor missing eigenmodes.

So far, we have established pseudospectral approximation results for the solution operator $T$. We now relate the pseudospectrum
of the original operator $-\mathcal L_u$ in
\eqref{eq:advectionDiffusionOperator} to that of the solution operator $T$ defined in \Cref{sec:convergence1}.

Set
\[
H:=L^2\Lambda^k(\Omega),\qquad
X:=\{v\in H\Lambda^{k}(\Omega): dv=0\},
\]
and let $P_X:H\to X$ be the $L^2$-orthogonal projection. Here and below, we use $\mathcal L_u$ to denote the unbounded operator on $X$
associated with the formulation in Section~\ref{sec:convergence1},
given by $\varepsilon d^{k-1}\delta^k+L_{\bm u}$ with domain
$D(\mathcal L_u)=T(H)\subset X$, using the opposite sign convention to \eqref{eq:advectionDiffusionOperator}. Then, by the definition of $T$ in \eqref{eq:gen:source},
\[
(\mathcal L_{ u}+\nu I)Tf= P_Xf,
\qquad \forall f\in H,
\]
for some $\nu >0$. For any $(\lambda,u)\in \mathbb C\times X \setminus\{0\}$ with $\lambda \neq-\nu$ is an eigenpair of $\mathcal{L}_u$ if and only if $((\lambda+\nu)^{-1},u)$ is an eigenpair of $T$. For any $\lambda\in\mathbb C, \lambda\neq -\nu$ such that $(\lambda I-\mathcal{L}_u)^{-1}\in \mathcal L(X;X)$, set $z = (\lambda +\nu)^{-1}$. Since $T$ is compact and $z\neq 0$, the eigenpair correspondence
above implies that $z\notin\sigma(T)$.
Hence, $(zI-T)^{-1}\in\mathcal L(H;H)$.

For any $f\in H$, set $u = (\lambda I-\mathcal{L}_u)^{-1}P_Xf\in D(\mathcal{L}_u)\subset X$. Then
\begin{align*}
    P_Xf &= (\lambda I-\mathcal{L}_u)u\\
    & = ((\lambda+\nu)I - (\mathcal{L}_u+\nu I))u.
\end{align*}
Applying \(T\) to both sides and using $u\in D(\mathcal{L}_u)$ and the definition of  $T$, we obtain
\begin{equation*}
    (\lambda+\nu)Tu - u = TP_Xf = Tf,
\end{equation*}
and hence
\begin{equation*}
    u = -z(z I-T)^{-1}Tf,\qquad (\lambda I-\mathcal{L}_u)^{-1}P_X = -z(z I-T)^{-1}T = zI -z^2(zI-T)^{-1},
\end{equation*}
where the left-hand side of the second equation is viewed as an operator on $H$. This implies that
\begin{equation}\label{eq:spec:auxi1}
    |z|^2\|(zI-T)^{-1}\|_{\mathcal{L}(H;H)} -|z|\leq \|(\lambda I-\mathcal{L}_u)^{-1}\|_{\mathcal{L}(X;X)}\leq |z|^2\|(zI-T)^{-1}\|_{\mathcal{L}(H;H)} +|z|.
\end{equation}

For any $0<\eta <|\lambda+\nu|$ and $$\lambda\in \sigma_{\eta}^X(\mathcal{L}_u): = \left\{\lambda\in \mathbb C : \|(\lambda I - \mathcal{L}_u)^{-1}\|_{\mathcal{L}(X;X)} > \eta^{-1}\right\},$$ it follows from \eqref{eq:spec:auxi1} that
\begin{equation*}
    z=(\lambda +\nu)^{-1}\in \sigma_{\tau_+}^H(T) :=\left\{\lambda\in \mathbb C : \|(\lambda I - T)^{-1}\|_{\mathcal{L}(H;H)} > \tau_+^{-1}\right\},
\end{equation*}
\begin{equation*}
    \tau_+ = \frac{\eta|z|^2}{1-\eta|z|} = \frac{\eta}{|\lambda+\nu|(|\lambda+\nu|-\eta)}.
\end{equation*}
Similarly, if
\begin{equation*}
    (\lambda+\nu)^{-1}\in \sigma_{\tau_-}^H(T),\quad \tau_- = \frac{\eta|z|^2}{1+\eta|z|} = \frac{\eta}{|\lambda+\nu|(|\lambda+\nu|+\eta)},\quad \eta>0,
\end{equation*}
then it follows from \eqref{eq:spec:auxi1} that $\lambda\in \sigma_{\eta}^X(\mathcal{L}_u)$. Thus, under the transformation $z=(\lambda+\nu)^{-1}$, the pseudospectrum of $\mathcal L_u$ is controlled from both sides by the pseudospectra of the compact solution operator $T$.

\section{Generalized advection-diffusion eigenvalue problem} \label{sec:convergence1}

\subsection{Formulations and finite element schemes}

Consider the advection-diffusion equation of $k$-forms in the space-time domain $\Omega\times [0,T]$ (cf. \cite{heumann2012fully,2016Pagliantini}): Find a time-dependent differential $k$-form $v(t)$ such that
\begin{align*}
    \partial_t v(t) + \varepsilon d^{k-1}\delta^{k} v(t) + L_uv(t) &= {f}(t),\\
    d^{k}v(t)&=0,
\end{align*}
where $u:\bar{\Omega}\times[0,T]\rightarrow \mathbb R^n$ is a given  velocity field, $f(t)\in \Lambda^{k}_{\mathbb C}(\Omega)$ a source term, and the scalar diffusivity parameter $\varepsilon$ is a non-negative constant.

With $n=3$ and $k=2$, this becomes the magnetic advection-diffusion equation \eqref{eq:MHD} with the divergence constraint $\nabla\cdot \bm B=0$.

Assume that $\varepsilon$ is a positive constant and $u=u(x)$ is time-independent, we now introduce the generalized advection-diffusion eigenvalue problem: Find a $k$-form eigenvector $v\ne0$ and an eigenvalue $\lambda\in \mathbb C$ such that
\begin{equation}\label{k-form-eigenvalue}
\begin{cases}
    \varepsilon d^{k-1}\delta^k v + L_u  v  = \lambda  v,\\
    d^k v=0.
    \end{cases}
\end{equation}
 For a nonzero eigenvalue $\lambda$, this problem is equivalent to
\begin{equation}\label{eq:gen:original}
    \varepsilon d^{k-1}\delta^k v + d^{k-1}\mathrm{i}_u^k v  = \lambda  v.
\end{equation}
From here on, we also write $\mathrm{i}_u$ for $\mathrm{i}_u^k$. Applying $d^{k}$ to \eqref{eq:gen:original} and using the fact that $d^{k}\circ d^{k-1}=0$, we have $d^{k}v=0$. The equivalence follows.

The mixed formulation of \eqref{k-form-eigenvalue} is as follows: Find $ v\in \mathring{H}^*\Lambda^k(\Omega)$ with $v\neq 0$ and $\lambda \in \mathbb C$ such that
\begin{equation}\label{eq:gen:mixed}
    \begin{aligned}
    \varepsilon(\delta^k v,\delta^k w) + (\mathrm{i}_u^k v,\delta^k w) + (\delta^{k+1} p, w) & =\lambda( v, w),\quad \forall  w\in  \mathring{H}^*\Lambda^k(\Omega),\\
    ( v,\delta ^{k+1} q)&=0,\quad \forall  q\in \mathring{H}^*\Lambda^{k+1}(\Omega),
\end{aligned}
\end{equation}
for some $p\in \mathring{H}^*\Lambda^{k+1}(\Omega)$. The discrete formulation of \eqref{eq:gen:mixed} is obtained by replacing the spaces with finite element spaces in a complex: Find $ v_h\in \Lambda_h^k$ with $v_h\neq 0$ and $\lambda_h\in\mathbb C$ such that
\begin{equation}\label{eq:gen:discrete_mixed}
    \begin{aligned}
    \varepsilon(\delta^k v_h,\delta^k w_h) + (\mathrm{i}_u^k v_h,\delta^k w_h) + (\delta^{k+1} p_h, w_h) & =\lambda_h( v_h, w_h),\quad \forall  w_h\in \Lambda_h^k,\\
    ( v_h,\delta ^{k+1} q_h)&=0,\quad \forall  q_h\in \Lambda_h^{k+1} ,
    \end{aligned}
\end{equation}
for some $p_h\in\Lambda_h^{k+1}$.

In the subsequent sections, we prove the convergence of the finite element problem \eqref{eq:gen:discrete_mixed}. The generalized advection-diffusion eigenvalue problem is non-symmetric. Therefore, our analysis is based on the use of Babu\v{s}ka–Osborn theory \cite{1991BabuskaOsborn}.

\subsection{Approximation and convergence}

In the remainder of this section, we analyze the convergence properties of the finite element scheme introduced above.
The analysis draws on techniques from finite element eigenvalue problems for the Maxwell equations~\cite{2010Boffi}.
However, the presence of the non-self-adjoint advection term and the possible competition between advection and diffusion
(as shown in the well-known difficulties of advection-dominated regimes) cause the analysis to deviate from the classical Maxwell case.

In this section, we assume that~$u \in L^{\infty}(\Omega)$.
Notably, we do not require the diffusion coefficient~$\varepsilon$ to be sufficiently large
(as is commonly assumed in source-type problems).
This flexibility arises because, in eigenvalue problems, one can always introduce an arbitrary shift of the form~$\nu v$.
Nevertheless, the constants appearing in the estimates below may still depend on~$\varepsilon$.

Moreover, we allow the domain $\Omega$ to have non-trivial topology and extend the analysis to include harmonic forms, which are often omitted in the literature for the Maxwell eigenvalue problems.

Define the kernel
\begin{equation*}
    \mathbb K^k:=\{ v\in \mathring{H}^*\Lambda^{k}(\Omega):( v,\delta^{k+1} p)=0,\forall  p\in\mathring{H}^*\Lambda^{k+1}(\Omega) \}
\end{equation*}
and the discrete kernel
\begin{equation}\label{def:Kh}
    \mathbb K_h^k:=\{ v_h\in\Lambda_h^k:( v_h,\delta^{k+1} p_h)=0,\forall  p_h\in \Lambda_h^{k+1}\}.
\end{equation}
Then the discrete mixed formulation \eqref{eq:gen:discrete_mixed} is equivalent to: Find $ v_h\in \mathbb K^k_h$ such that
\begin{equation*}
    \varepsilon(\delta^k v_h, \delta^k w_h) + (\mathrm{i}_u^k v_h,\delta^{k} w_h) = \lambda_h ( v_h, w_h),\quad \forall  w_h\in \mathbb K^k_h.
\end{equation*}

As in the mixed form of the Maxwell eigenvalue problem \cite{2007Boffi,2010Boffi}, we establish the strong discrete compactness of the discrete kernel, as well as the weak and strong approximation properties. The main difference from the Maxwell case is due to the influence of the advection term on the regularity of the source problem, which further affects the compactness argument.

The following result is a generalization of the Maxwell result, incorporating non-trivial cohomology (harmonic forms).
\begin{lemma}[Strong discrete compactness]\label{lem:str_dis_com}
     For any sequence $\{ u_{h_n}\}$ with $ u_{h_n}\in \mathbb K_{h_n}^k$, $h_n\rightarrow 0$, which is uniformly bounded in $H^*\Lambda^k$, there exists a subsequence of $\{ u_{h_n}\}$ converging strongly in $L^2\Lambda^k(\Omega)$ to a limit $ u$ with $( u,\delta^{k+1} q)=0$ for all $ q\in \mathring{H}^*\Lambda^{k+1}(\Omega)$.
\end{lemma}
\begin{proof}
    There exists an $L^2$-orthogonal decomposition ${ u}_{h_n} = \tilde{ u}_{h_n}  +  p_{h_n}$ with $ p_{h_n}\in \H_{h_n}^k,\tilde{ u}_{h_n}\in \mathbb K_{h_n}^k\cap (\H_{h_n}^k)^{\perp}$ and $(\tilde{ u}_{h_n},q)= 0$ for any $q\in \H_{h_n}^k$. From Lemma \ref{lem:discrete_har:appro}, there exists $ q_{n}\in \H^k$ such that $\| q_{n}\|_{L^2}\leq \| p_{h_n}\|_{L^2}$ and $\| p_{h_n}- q_{n}\|_{L^2}\leq Ch^s_{h_n}\| p_{h_n}\|_{L^2}$. Hence, $\{ q_{n}\}$ is uniformly bounded. Since $\H^k$ is a finite-dimensional Hilbert space, we can assume that there exists a subsequence (still denoted by $\{ q_{n}\}$) such that the corresponding sequence $ q_{n}\rightarrow  q\in \H^k$ in $L^2$. It also holds that $ p_{h_n}\rightarrow  q\in \H^k$ in $L^2$.

On the orthogonal complement of (discrete) harmonic forms, the rest of the analysis follows from that of the Maxwell equation.

Define $ u(n)\in  \mathring{H}^*\Lambda^{k}(\Omega)$, $u(n)\perp \H^k$  via a Hodge mapping \cite{2002Hiptmair} by requiring
$$\delta^k  u(n)= \delta^k \tilde{ u}_{h_n}, \quad ( u(n),\delta^{k+1} p) =0, \quad \forall p\in \mathring{H}^*\Lambda^{k+1}(\Omega).$$
We have $du(n)=0$ and $d_{h}\tilde{ u}_{h_n}=0$, where $d_{h}$ is the $L^{2}$ adjoint of $\delta$. Therefore there exist $\psi_{h_n}\in \delta^k\Lambda_{h_n}^k$ and $\psi (n)\in \delta^{k} \mathring{H}^*\Lambda^{k}(\Omega)$ such that
\begin{align*}
    (\tilde{ u}_{h_n}-  u(n), \sigma) = (\psi_{h_n}-\psi(n),\delta^{k} \sigma),\quad \forall  \sigma\in \Lambda_{h_n}^k.
\end{align*}
Choosing $ \sigma = \pi_{h_n}^k u(n) - \tilde{ u}_{h_n}$,  we have
\begin{align*}
    (\tilde{ u}_{h_n}-  u(n),\pi_{h_n}^k u(n) - \tilde{ u}_{h_n}) &= (\psi_{h_n}-\psi(n),\delta^{k}(\pi_{h_n}^k u(n) - \tilde{ u}_{h_n}))\\
    & = (\psi_{h_n}-\psi(n),\pi_{h_n}^{k-1}\delta^k u(n)- \delta^{k}\tilde{ u}_{h_n}) = 0.
    \end{align*}
Consequently,
    \begin{equation*}
        \| u(n) - \tilde{ u}_{h_n}\|_{L^2}\leq \| u(n) -\pi_{h_n}^k u(n)\|_{L^2}\leq Ch_{n}^s\| u(n)\|_{H^s}\leq Ch_{n}^s\|\delta^k\tilde{ u}_{h_n}\|_{L^2}.
    \end{equation*}
    Here we used that $d^k u(n)=0$ and $(u(n),q) = 0$, $\forall q\in \H^k$, and \eqref{eq:s-regular}. Recall that $\{ u(n)\}$ is a uniformly bounded sequence in $\mathbb K^k$. Since $\mathbb K^k$ is compactly embedded in $L^2\Lambda^k(\Omega)$, there exists a subsequence (still denoted by $\{ u(n)\}$) converging strongly to some limit $ u$ in $L^2$ with $( u,\delta^{k+1} w)= 0$ for all $ w\in  \mathring{H}^*\Lambda^{k+1}(\Omega)$. Hence the discrete subsequence $\{ u_{h_n}\}$  converges to  $ u + q$ in $L^2$, which completes the proof.
\end{proof}
\begin{lemma}[Weak approximation]
    There exists $\omega_1(h)$, tending to zero as $h$ goes to zero, such that for every $ q\in \mathring{H}^*\Lambda^{k+1}(\Omega)$
    \begin{equation*}
        \sup_{ v_h\in \mathbb K_h^k}\frac{|( v_h, \delta^{k+1} q)|}{\| v_h\|_{H^*\Lambda}} \leq \omega_1(h)\|\delta^{k+1} q\|_{L^2}.
    \end{equation*}
\end{lemma}
\begin{proof}
The proof follows a similar argument as in \cite[Proposition 5]{2007Boffi}.
\end{proof}

\begin{lemma}\label{lem:coercivity}
 Assume that $u\in  L^{\infty}(\Omega)$.    There exist positive constants $\nu\geq \frac{C\|u\|_{L^{\infty}}^2}{2\varepsilon} + 1>0$ and $\alpha = \min\{\frac{\varepsilon}{4},\frac{1}{2}\}>0$ such that for any $v\in \Lambda^{k}_{\mathbb C}(\Omega)$, $w\in \Lambda^{k-1}_{\mathbb C}(\Omega)$ and $0\leq k\leq n$, we have
\begin{equation}\label{eq:shift_coe}
    |\varepsilon( w, w) + (\mathrm{i}_u^k v,w) + \nu ( v, v)|\geq \alpha (\| v\|_{L^2} + \|w\|_{L^2})^2.
\end{equation}
\end{lemma}
\begin{proof}
There exists a positive constant $C>0$ such that
    \begin{align*}
        |\varepsilon( w, w) + (\mathrm{i}_u^k v,w) + \nu ( v, v)|& \geq |\varepsilon( w, w) + \nu ( v, v)| - |(\mathrm{i}_u^k v,w)|\\
        &\geq \varepsilon\|w\|_{L^2}^2+\nu\|v\|_{L^2}^2-C^{\frac{1}{2}}\|u\|_{L^{\infty}}\|v\|_{L^2}\|w\|_{L^2}.
    \end{align*}
Choosing $\nu \geq \frac{C\|u\|_{L^{\infty}}^2}{2\varepsilon} + 1$, we have
    \begin{align*}
        |\varepsilon( w, w) + (\mathrm{i}_u^k v,w) + \nu ( v, v)|\geq \frac{\varepsilon}{2}\|w\|_{L^2}^2 + \|v\|_{L^2}^2\geq \alpha(\| v\|_{L^2} + \|w\|_{L^2})^2.
    \end{align*}
\end{proof}
Define a sesquilinear form $a:\mathring{H}^*\Lambda^{k}(\Omega)\times \mathring{H}^*\Lambda^{k}(\Omega)\rightarrow \mathbb C$ by
\begin{equation}\label{def:a}
    a( v, w) = \varepsilon(\delta^k v,\delta^k w) + (\mathrm{i}_u^k v,\delta^k w) + \nu ( v, w).
\end{equation}
Note that
\begin{equation*}
    |a(v,w)|\leq C\|v\|_{H^*\Lambda}\|w\|_{H^*\Lambda},\quad \forall v,w\in \mathring{H}^*\Lambda^{k}(\Omega),
\end{equation*}
and
\begin{equation*}
    |a(v,v)| = |\varepsilon( \delta^kv, \delta^kv) + (\mathrm{i}_u^k v,\delta^k v) + \nu ( v, v)|\geq\alpha\|v\|_{H^*\Lambda}^2,\quad \forall v\in \mathring{H}^*\Lambda^{k}(\Omega).
\end{equation*}
Then we can define the solution operator of the source problem $T: L^2\Lambda^k(\Omega)\rightarrow \mathbb K^k$ by
\begin{equation}\label{eq:gen:source}
    a(T f, w) = ( f, w),\quad \forall  w\in \mathbb{K}^k.
\end{equation}
By the Lax-Milgram theorem, the operator $T$ is well-defined.
Since the embedding $\mathbb K^k \subset\mathring{H}^*\Lambda^{k}(\Omega)\cap H\Lambda^k(\Omega) \hookrightarrow L^2\Lambda^k(\Omega)$ is compact, $T: L^2\Lambda^k(\Omega) \rightarrow L^2\Lambda^k(\Omega)$ is a compact operator. Similarly,   define $T_h: L^2\Lambda^k(\Omega)\rightarrow \mathbb K_h^k$ by
\begin{equation}\label{eq:gen:discrete_source}
    a(T_h f, w_h) = ( f, w_h),\quad \forall w_h\in \mathbb{K}_h^k.
\end{equation}

    Denote by $V_0:=T(L^2\Lambda^k(\Omega))\subset \mathring{H}^*\Lambda^{k}(\Omega)$ the solution space consisting of all possible solutions  of \eqref{eq:gen:source}. The space $V_0$ will be endowed with its natural norm $\| v\|_{V_0}:=\inf_{ f\in L^2\Lambda^k(\Omega),T f= v}\| f\|_{L^2}$.  We need the following to prove the strong approximation.
\begin{lemma}[Compact embedding]\label{lem:elliptic_reg}
Assume that the vector field $u\in L^{\infty}$. Then $V_0$ is compactly embedded in $\mathring{H}^*\Lambda^{k}(\Omega)$.
\end{lemma}
\begin{proof}
Let $\{f_j\}_{j\geq1}$ be bounded in $L^2\Lambda^k(\Omega)$,
and let $(v_j,p_j)$ be the solutions of the following problem:
\begin{align*}
    \varepsilon(\delta^k v_j,\delta^k w) + (\mathrm{i}_u^k v_j,\delta^k w) + (\delta^{k+1} p_j, w) +\nu (v_j,w)& =(f_j,w),\quad \forall  w\in  \mathring{H}^*\Lambda^k(\Omega),\\
    ( v_j,\delta ^{k+1} q)&=0,\quad \forall  q\in \mathring{H}^*\Lambda^{k+1}(\Omega).
\end{align*}
By the stability estimate and the continuous embedding
$\mathbb K^k\hookrightarrow H^s\Lambda^k(\Omega)$, for some $s> \frac{1}{2}$,
\[
\|v_j\|_{H^s}\leq C\|f_j\|_{L^2}.
\]
Since $H^s\Lambda^k(\Omega)$ is compactly embedded in $L^2\Lambda^k(\Omega)$, we can extract a subsequence, still denoted by $v_j$ for simplicity:
\[
v_j\rightarrow v
\quad\text{strongly in }L^2.
\]

For $j,m\geq1$, set
\[
e_{jm}:=v_j-v_m,\qquad
r_{jm}:=p_j-p_m,\qquad
g_{jm}:=f_j-f_m.
\]
Subtracting the two mixed formulations and testing with $e_{jm}$ yields
\[
|\varepsilon(\delta^k e_{jm},\delta^k e_{jm})
+(\mathrm i_u^k e_{jm},\delta^k e_{jm})
+\nu( e_{jm},e_{jm})|
=|(g_{jm},e_{jm})|,
\]
where the multiplier term vanishes since
\[
(e_{jm},\delta^{k+1}r_{jm})=0.
\]
Hence for some $\alpha>0$,
\begin{equation*}
    \alpha\|e_{jm}\|_{H^*\Lambda}^2\leq \|g_{jm}\|_{L^2}\|e_{jm}\|_{L^2}.
\end{equation*}
Since $\{f_j\}$ is bounded and $\{v_j\}$ is a Cauchy sequence in $L^2$,
the right-hand side tends to zero as $j,m\to\infty$. Therefore
\[
\| v_j-v_m\|_{H^*\Lambda}\longrightarrow0.
\]
Consequently, $\{v_j\}$ is a Cauchy sequence in
$\mathring H^*\Lambda^k(\Omega)$, and hence has a strongly
convergent subsequence in this space.
\end{proof}

\begin{lemma}[Strong approximation] There exists $\omega_2(h)$, tending to zero as $h$ goes to zero, such that for any $ u \in V_0$, there exists $ u_h\in \mathbb K_h^k$ such that
    \begin{equation*}
        \| u- u_h\|_{H^*\Lambda} \leq \omega_2(h)\| u\|_{V_0}.
    \end{equation*}
\end{lemma}
\begin{proof}
Using the compactness result of Lemma \ref{lem:elliptic_reg}, the proof follows a similar argument as in \cite[Proposition 6]{2007Boffi}.
\end{proof}

\begin{theorem}\label{thm:norm_conv:gen}
    It holds that
    \begin{equation*}
        \|T-T_h\|_{\mathcal{L}(L^2\Lambda^k(\Omega);H^*\Lambda^k(\Omega))} \rightarrow 0,\quad \text{ as }h\rightarrow 0.
    \end{equation*}
\end{theorem}
\begin{proof}
The proof follows a similar argument as in \cite[Theorem 13.4]{2010Boffi}.
\end{proof}

\begin{remark}
The framework that we follow in this paper requires proving that the solution operator maps $L^2\Lambda^k$ compactly into $H^*\Lambda^k$. This yields the operator convergence
\[
\|T-T_h\|_{\mathcal{L}(L^2\Lambda^k(\Omega);H^*\Lambda^k(\Omega))}\to 0,
\]
which in turn immediately implies
\[
\|T-T_h\|_{\mathcal{L}(L^2\Lambda^k(\Omega);L^2\Lambda^k(\Omega))}\to 0.
\]
\end{remark}
Combining Theorem~\ref{thm:norm_conv:gen} with \eqref{eq:spec:auxi1} yields a two-sided approximation result for the pseudospectra of the operator $\mathcal{L}_u$ itself, rather than only for those of its solution operator $T$. Let $X_h:=\mathbb{K}^k_h$, and let $P_{X_h}:H\rightarrow X_h$ be the $L^2$-orthogonal projection. Since the discrete source problem \eqref{eq:gen:discrete_source} is uniquely solvable, $T_h|_{X_h}:X_h\rightarrow X_h$ is an isomorphism. We therefore define $\mathcal{L}_{u,h}: X_h\rightarrow X_h$ by $(\mathcal{L}_{u,h}+\nu I)T_hf=P_{X_h}f$ for all $f\in H$. Since $a(T_hf,w_h)=(f,w_h)$ for all $w_h\in X_h$, the argument leading to \eqref{eq:spec:auxi1} applies verbatim with $(\mathcal{L}_u,T,X)$ replaced by $(\mathcal{L}_{u,h},T_h,X_h)$. We use the notation $\sigma^{X_h}_{\eta}(\mathcal{L}_{u,h})$ and $\sigma^{H}_{\tau}(T_h)$ as in Section~\ref{sec:pseudospectra}.

\begin{corollary}[Approximation of the pseudospectra of the operator]\label{cor:pseudospectra:operator}
    Let $0<a\leq b<\infty$, and set $K:=\{\lambda\in\mathbb{C}: a\leq|\lambda+\nu|\leq b\}$. Let $0<\eta<\eta'$ satisfy $a(\eta'-\eta)>2\eta\eta'$. Set
    \begin{equation*}
        \tau_\pm(\eta,\lambda):=\frac{\eta}{|\lambda+\nu|\bigl(|\lambda+\nu|\mp\eta\bigr)},\qquad \eta^*:=\min_{\lambda\in K}\bigl(\tau_-(\eta',\lambda)-\tau_+(\eta,\lambda)\bigr).
    \end{equation*}
    Then $\eta^*>0$ and, whenever $\|T-T_h\|_{\mathcal{L}(H;H)}<\eta^*$,
    \begin{equation}\label{eq:pseudospectra:operator}
        \sigma^{X_h}_{\eta}(\mathcal{L}_{u,h})\cap K\subseteq\sigma^{X}_{\eta'}(\mathcal{L}_u)
        \qquad\text{and}\qquad
        \sigma^{X}_{\eta}(\mathcal{L}_u)\cap K\subseteq\sigma^{X_h}_{\eta'}(\mathcal{L}_{u,h}).
    \end{equation}
    In particular, by Theorem~\ref{thm:norm_conv:gen}, \eqref{eq:pseudospectra:operator} holds for all sufficiently small $h$.
\end{corollary}
\begin{proof}
    With $\alpha:=|\lambda+\nu|$, a direct computation gives
    \begin{equation*}
        \tau_-(\eta',\lambda)-\tau_+(\eta,\lambda)=\frac{\alpha(\eta'-\eta)-2\eta\eta'}{\alpha(\alpha+\eta')(\alpha-\eta)}.
    \end{equation*}
    The assumption $a(\eta'-\eta)>2\eta\eta'$ implies $\eta<a/2$, hence the above quantity is continuous and positive for $a\leq \alpha\leq b$, and therefore $\eta^*>0$.

    Let $\lambda\in\sigma^{X_h}_{\eta}(\mathcal{L}_{u,h})\cap K$ and set $z=(\lambda+\nu)^{-1}$. If $\lambda\in\sigma(\mathcal{L}_{u,h})$, then $z\in\sigma(T_h)$ by the eigenpair correspondence. Otherwise, since $\eta<\alpha$, the upper bound in the discrete counterpart of \eqref{eq:spec:auxi1} gives $\|(zI-T_h)^{-1}\|_{\mathcal{L}(H;H)}>\tau_+(\eta,\lambda)^{-1}$. In both cases $z\in\sigma^{H}_{\tau_+(\eta,\lambda)}(T_h)$, and Theorem~\ref{thm:pseudospectra:perturb} yields
    \begin{equation*}
        z\in\sigma^{H}_{\tau_+(\eta,\lambda)+\|T-T_h\|_{\mathcal{L}(H;H)}}(T)\subseteq\sigma^{H}_{\tau_-(\eta',\lambda)}(T).
    \end{equation*}
    If $\lambda\in\sigma(\mathcal{L}_u)$, there is nothing to prove. Otherwise, the lower bound in \eqref{eq:spec:auxi1} gives
    \begin{equation*}
        \|(\lambda I-\mathcal{L}_u)^{-1}\|_{\mathcal{L}(X;X)}\geq|z|^2\|(zI-T)^{-1}\|_{\mathcal{L}(H;H)}-|z|>|z|^2\tau_-(\eta',\lambda)^{-1}-|z|={\eta'}^{-1},
    \end{equation*}
    that is, $\lambda\in\sigma^{X}_{\eta'}(\mathcal{L}_u)$. The second inclusion in \eqref{eq:pseudospectra:operator} follows by exchanging the roles of $(\mathcal{L}_u,T,X)$ and $(\mathcal{L}_{u,h},T_h,X_h)$.
\end{proof}

Before concluding this section, we introduce the a priori error estimate of the advection-diffusion eigenvalue problem. Following eigenvalue analysis, we consider the solution operators of the adjoint problem: $T^*:L^2\Lambda^k(\Omega)\rightarrow \mathbb K^k$ defined by
\begin{equation}\label{eq:gen:source_dual}
    a( w,T^* f) = ( w, f),\quad \forall  w\in \mathbb{K}^k,
\end{equation}
and $T_h^*:L^2\Lambda^k(\Omega)\rightarrow \mathbb K_h^k$ defined by
\begin{equation}\label{eq:gen:discrete_source_dual}
    a( w_h,T_h^* f) = ( w_h, f),\quad \forall w_h\in \mathbb{K}_h^k.
\end{equation}

Let $\mu$ be a nonzero eigenvalue of $T$ with algebraic multiplicity $m$. Then $\bar{\mu}$ will be an eigenvalue of $T^*$ with the same algebraic multiplicity $m$. Let  $E$ and $E^*$ be the spaces of generalized eigenvectors of $T$ and $T^*$ corresponding to $\mu$ and $\bar \mu$, respectively.
 Then, from Theorem \ref{thm:norm_conv:gen}, there exist $m$ eigenvalues of $T_h$, denoted by $\mu_{1h},\cdots,\mu_{mh}$, converging to $\mu$. Therefore the following holds   (cf. \cite{1981MercierOsborn,2010Boffi}).
\begin{lemma}\label{lem:eigenvalue_err}
    There exist two constants $C>0$ and $h_0>0$ such that for any $h<h_0$
    \begin{equation*}
        |\mu^{-1}-\frac{1}{m}\sum_{i=1}^m\mu_{ih}^{-1}| \leq C\big( \sup_{\substack{ u\in E\\ \| u\|_{L^2} =1}}\sup_{\substack{ v\in E^*\\ \| v\|_{L^2} =1}}| ((T-T_h) u, v)| + \|(T-T_h)_{/E}\|\|(T^*-T_h^*)_{/E^*}\|\big),
    \end{equation*}
    where
    $$
    \|(T-T_h)_{/E}\|:= \sup_{ v\in E,\| v\|_{L^2}=1}\|(T-T_h) v\|_{L^2}, \quad \|(T^*-T_h^*)_{/E^*}\|:= \sup_{ v\in E^*,\| v\|_{L^2}=1}\|(T^*-T_h^*)v \|_{L^2}.
    $$
\end{lemma}

To derive the convergence result, we consider the source problem without the $d$-free constraint. Define the solution operator $\tilde{T}: L^2\Lambda^k(\Omega)\rightarrow \mathring{H}^*\Lambda^{k}(\Omega)$ by
\begin{equation}\label{tilde-T-problem}
    a(\tilde{T} f, w) = ( f, w),\quad \forall  w\in \mathring{H}^*\Lambda^{k}(\Omega),
\end{equation}
and define $\tilde{T}_h  :L^2\Lambda^k(\Omega)\rightarrow \Lambda_h^k$ by
\begin{equation*}
    a(\tilde{T}_h f, w) = ( f, w),\quad \forall  w\in \Lambda_h^k.
\end{equation*}
Recall that the difference between $\tilde{T}$, $\tilde{T}_{h}$ and $T$, $T_{h}$ is that the latter operators map to $\mathbb K^k$ and $\mathbb K_{h}^k$.

Similarly, define the solution operator $\tilde{T}^* :L^2\Lambda^k(\Omega)\rightarrow \mathring{H}^*\Lambda^{k}(\Omega)$ of the adjoint problem by
\begin{equation*}
    a( w,\tilde{T}^* f) = ( w, f),\quad \forall  w\in \mathring{H}^*\Lambda^{k}(\Omega),
\end{equation*}
and $\tilde{T}^*_h :L^2\Lambda^k(\Omega)\rightarrow \Lambda_h^k$  by
\begin{equation*}
    a( w,\tilde{T}^*_h f) = ( w, f),\quad \forall  w\in \Lambda_h^k.
\end{equation*}
\begin{theorem}\label{thm:eigenvalue_err}
    Assume  the following approximation properties hold:
    \begin{equation}\label{assumption:r1}
        \sup_{\substack{ v\in E\\ \| v\|_{L^2} =1}} \| (\tilde{T}   - \tilde{T}_h)  v\|_{H^*\Lambda}\leq Ch^{r_1}
    \end{equation}
    and
    \begin{equation}\label{assumption:r2}
        \sup_{\substack{ v\in E^*\\ \| v\|_{L^2} =1}} \| (\tilde{T}^*   - \tilde{T}_h^*)  v\|_{H^*\Lambda}\leq Ch^{r_2}.
    \end{equation}
    Here, the rates $r_{1}, r_{2}>0$ and the constant $C>0$ may depend on the polynomial degrees of the discrete finite element spaces and the regularity of the generalized eigenvector spaces $E$ and $E^*$, the constant $\varepsilon$, the velocity field $ u$, but are independent of $h$. Furthermore, define $\Pi_{\delta}^k: L^2\Lambda^k(\Omega)\rightarrow \delta^{k+1}\mathring{H}^*\Lambda^{k+1}(\Omega)$ and $\Pi_{\delta,h}^k:L^2\Lambda^k(\Omega)\rightarrow\delta^{k+1}\Lambda_h^{k+1}(\Omega)$ as $L^2$-orthogonal projections to the range of $\delta$.
    Assume that
    \begin{equation}\label{assumption:r3}
        \sup_{\substack{ v\in E^*\\ \| v\|_{L^2} =1}}\| (I-\Pi_{\delta,h}^k)\Pi_{\delta}^k \tilde{T}^* v\|_{L^2}\leq Ch^{r_3}.
    \end{equation}
    Then we have
    \begin{equation*}
        |\mu^{-1}-\frac{1}{m}\sum_{i=1}^m\mu_{ih}^{-1}|\leq Ch^{{r_{1}+\min\{r_{2}, r_{3}\}}}.
    \end{equation*}
\end{theorem}
\begin{proof}
\emph{First}, we estimate the term
\begin{equation*}
    \|(T-T_h)_{/E}\|= \sup_{\substack{ v\in E\\ \| v\|_{L^2}=1}}\|(T-T_h) v\|_{L^2}.
\end{equation*}
For any $ f\in \mathbb K^k$, the solution $T f = \tilde{T} f$, i.e.,
\begin{equation*}
    a(T f, w) = ( f, w),\quad \forall  w\in \mathring{H}^*\Lambda^{k}(\Omega),
\end{equation*}
since for the problem \eqref{tilde-T-problem}, $d$-free right-hand sides imply $d$-free solutions. Likewise, $T_h f= \tilde{T}_h  f$. Then $(T-T_h) u = (\tilde{T}-\tilde{T}_h) u$ for any $ u\in E$ and
\begin{equation*}
       \sup_{\substack{ v\in E\\ \| v\|_{L^2}=1}}\|(T-T_h) v\|_{L^2}\leq  \sup_{\substack{ v\in E\\ \| v\|_{L^2} =1}} \| (T - T_h  )v\|_{H^*\Lambda}\leq Ch^{r_{1}}.
\end{equation*}

\emph{Second},  consider the term
\begin{equation*}
    \|(T^*-T_h^*)_{/E^*}\|:= \sup_{ \substack{v\in E^*\\ \| v\|_{L^2}=1}}\|(T^*-T_h^*) v\|_{L^2}.
\end{equation*}
Recall that $T^* f$ satisfies
\begin{equation*}
    a( w,T^* f)  =( w, f),\quad \forall  w \in \mathbb{K}^k,
\end{equation*}
and
\begin{equation*}
    a( w,\tilde{T}^* f) = ( w, f),\quad \forall  w\in \mathring{H}^*\Lambda^{k}(\Omega).
\end{equation*}
Set $\tilde{ u} = \tilde{T}^* f-\Pi_{\delta}^k \tilde{T}^* f \in \mathbb{K}^k$, since $a(w,\delta^{k+1}q) = \nu(w,\delta^{k+1}q) = 0$ holds for any $w\in \mathbb{K}^k$ and $q\in \mathring{H}^*\Lambda^{k+1}(\Omega)$, we have
\begin{equation*}
    a( w,\tilde{ u})  = a( w,\tilde{T}^* f) - a(w,\Pi_{\delta}^k \tilde{T}^* f)
     = (w,f),\quad \forall  w\in\mathbb{K}^k.
\end{equation*}
This implies that for any $ f\in L^2\Lambda^k(\Omega)$
\begin{equation*}
    T^* f= \tilde{T}^* f-\Pi_{\delta}^k \tilde{T}^* f,
\end{equation*}
and, by a similar argument
\begin{equation*}
    T^*_h f= \tilde{T}^*_h f-\Pi_{\delta,h}^k \tilde{T}^*_h f.
\end{equation*}
Therefore,
\begin{align*}
    \|T^* v-T_h^* v\|_{L^2}&\leq \|\tilde{T}^* v-\tilde{T}^*_h v\|_{L^2} + \|\Pi_{\delta}^k \tilde{T}^* v-\Pi_{\delta,h}^k \tilde{T}^*_h v\|_{L^2}\\
    &\leq \|\tilde{T}^* v-\tilde{T}^*_h v\|_{L^2} + \|\Pi_{\delta}^k \tilde{T}^* v -\Pi_{\delta,h}^k \tilde{T}^* v\|_{L^2} + \|\Pi_{\delta,h}^k (\tilde{T}^* v - \tilde{T}^*_h v)\|_{L^2}\\
    &\leq 2\|\tilde{T}^* v-\tilde{T}^*_h v\|_{L^2}+ \|(I-\Pi_{\delta,h}^k)\Pi_{\delta}^k \tilde{T}^* v\|_{L^2},
\end{align*}
which implies that
\begin{equation*}
        \sup_{\substack{ v\in E^*\\ \| v\|_{L^2} =1}} \| (T^*   - T_h^*)  v\|_{L^2}\leq Ch^{\min\{r_{2}, r_{3}\}}.
    \end{equation*}

\emph{Finally}, we need to estimate the term
\begin{equation*}
    \sup_{\substack{ u\in E\\ \| u\|_{L^2} =1}}\sup_{\substack{ v\in E^*\\ \| v\|_{L^2} =1}}| ((T-T_h) u, v)|.
\end{equation*}
Note that
\begin{align*}
    | ((T-T_h) u, v)| &= |a((T-T_h) u,\tilde{T}^* v)| = |a(\tilde{T} u-\tilde{T}_h u,\tilde{T}^* v - \tilde{T}^*_h v)|\\
    &\leq C\|\tilde{T} u-\tilde{T}_h u\|_{H^*\Lambda}\|\tilde{T}^* v - \tilde{T}^*_h v\|_{H^*\Lambda},
\end{align*}
where we used $(T-T_h) u = (\tilde{T}-\tilde{T}_h) u$ for any $ u\in E$ and
\begin{equation*}
    a((\tilde{T}-\tilde{T}_h) u, v_h)= 0,\quad \forall  v_h\in \Lambda_h^k.
\end{equation*}
Hence
\begin{equation*}
    \sup_{\substack{ u\in E\\ \| u\|_{L^2} =1}}\sup_{\substack{ v\in E^*\\ \| v\|_{L^2} =1}}| ((T-T_h) u, v)|\leq Ch^{{r_{1}+\min\{r_{2}, r_{3}\}}}.
\end{equation*}
The result now follows from Lemma \ref{lem:eigenvalue_err}.
\end{proof}
The above theorem and its proof highlight the difference between the non-normal advection-diffusion problems studied in this paper and the self-adjoint Hodge Laplacian problems. In particular, the assumption \eqref{assumption:r2} involves estimates of the best approximation to the adjoint problem, which in general is not a direct consequence of the assumption \eqref{assumption:r1} for the original problem. Similarly, an additional assumption \eqref{assumption:r3} involves the adjoint problem. For self-adjoint problems such as the Maxwell equations, \eqref{assumption:r2} and \eqref{assumption:r3} can be dropped. However, due to the non-normality of our problem, these two conditions are required.

\begin{remark}[Alternative formulations]\label{rmk:alternative-formulation}
The Maxwell eigenvalue problem
\begin{eqnarray}
\label{curl-curl}
\nabla\times \nabla\times \bm B&=\lambda\bm B, \
\nabla\cdot \bm B&=0,
\end{eqnarray}
admits several equivalent formulations. For example, one may apply an additional curl operator and solve for $\bm j=\nabla\times \bm B$, which satisfies the same eigenvalue problem as $\bm B$. Alternatively, one may introduce a magnetic potential $\bm A$ satisfying $\bm B=\nabla\times  \bm A$ (ignoring harmonic forms) and impose the gauge condition $\nabla\cdot \bm A=0$. Then $\bm A$ also satisfies the $\curl$–$\curl$ system \eqref{curl-curl}. These formulations, together with their finite element discretizations based on de~Rham complexes, are all equivalent \cite{2010Boffi}. This equivalence extends more generally to half Hodge–Laplacian problems.

It is well known that the MHD eigenvalue (or source) problem can be reformulated in terms of a vector potential with an appropriate gauge condition \cite{heumann2012fully}. More precisely, the evolution problem \eqref{eq:MHD} can be rewritten as
\begin{eqnarray}\label{A-formulation}
\bm A_t-\bm u\times \nabla\times \bm A+ R_m^{-1}\nabla\times \nabla\times \bm A+\nabla p&=0, \
\nabla\cdot \bm A&=0.
\end{eqnarray}
Here, $p$ acts as a Lagrange multiplier enforcing the divergence-free gauge condition on $\bm A$, and \eqref{A-formulation} is precisely a half-Hodge formulation of advection–diffusion for the 1-form $\bm A$. However, for the MHD eigenvalue problem, and more generally for advection–diffusion problems of differential forms, the equivalence between discrete formulations can be lost. First, the cohomology information is not preserved when one applies an exterior derivative or passes to a potential formulation. Second, the Lie derivative of a finite element differential form need not belong to the finite element space. Therefore suitable projections are required in order to obtain equivalent finite element formulations.

\end{remark}

\section{Full advection-diffusion eigenvalue problem} \label{sec:convergence2}
In the previous section, we studied the advection--diffusion eigenvalue problem involving the half Hodge Laplacian.
This framework finds an important application in MHD, where it corresponds to the evolution of 2-forms.
However, certain relevant cases fall outside that setting, for instance, the 0-form case or the augmented MHD formulation
that incorporates the term~$-\grad \div \bm B$.
In the present section, we turn to the advection--diffusion eigenvalue problem with the \emph{full} Hodge Laplacian.
As we shall see, this problem differs in several essential ways from the problem analyzed earlier.

\subsection{Continuous and discrete formulations}
Consider the full Hodge Laplacian advection-diffusion eigenvalue problem: Find a $k$-form $ v$ and an eigenvalue $\lambda\in \mathbb C$ such that
\begin{equation*}
    \varepsilon(d^{k-1}\delta^k+\delta^{k+1}d^k) v + L_u v = \lambda v
\end{equation*}
with a constant $\varepsilon>0$.

The weak formulation is to find $ v\in  H\Lambda^k(\Omega)\cap \mathring{H}^*\Lambda^{k}(\Omega) $ and $\lambda \in \mathbb C$ such that
\begin{equation*}
    \varepsilon (\delta^k v,\delta^k w) + \varepsilon (d^k  v,d^k w) + (\mathrm{i}_u^{k+1}d^k v, w) + (\mathrm{i}_u^k v,\delta^k w) = \lambda( v, w),\quad \forall  w\in  H\Lambda^k(\Omega)\cap \mathring{H}^*\Lambda^{k}(\Omega).
\end{equation*}
This can be viewed as a generalization of the Hodge Laplacian eigenvalue problem (see \cite[Section~8]{2006ArnoldFalkWinther}, \cite[Section~17]{2010Boffi}). As in the case of Hodge Laplacian, the primal formulation is not well-suited for numerical approximation. Therefore, we consider a mixed formulation, which can be regarded as an extension of the Kikuchi formulation (cf. \cite{1987Kikuchi}, \cite[(17.2)]{2010Boffi}).

Define $ p = d^k v$. The mixed formulation of the eigenvalue problem is as follows: Find $ v\in \mathring{H}^*\Lambda^{k}(\Omega)$, $ p\in \mathring{H}^*\Lambda^{k+1}(\Omega)$ with $v\neq 0$ and $\lambda\in\mathbb C$ such that
\begin{equation}\label{eq:Hodge:mixed}
    \begin{aligned}
        \varepsilon (\delta^k v,\delta^k w) + \varepsilon(\delta^{k+1} p, w)+ (\mathrm{i}_u^{k+1} p, w) + (\mathrm{i}_u^k v,\delta^k w) &= \lambda( v, w),\quad \forall  w\in \mathring{H}^*\Lambda^{k}(\Omega),\\
    ( p, q) - ( v,\delta^{k+1} q) &=0,\quad \forall  q\in \mathring{H}^*\Lambda^{k+1}(\Omega).
    \end{aligned}
\end{equation}
The discrete version is as follows: Find $ v_h\in \Lambda_h^k$, $ p_h\in \Lambda_h^{k+1}$ with $v_h\neq 0$ and $\lambda_h\in \mathbb C$ such that
\begin{equation}\label{eq:Hodge:mixed:dis}
    \begin{aligned}
    \varepsilon (\delta^k v_h,\delta^k w_h) + \varepsilon(\delta^{k+1} p_h, w_h)+ (\mathrm{i}_u^{k+1} p_h, w_h) + (\mathrm{i}_u^k v_h,\delta^k w_h) &= \lambda_h( v_h, w_h),\quad \forall  w_h\in \Lambda_h^k,\\
    ( p_h, q_h) - ( v_h,\delta^{k+1} q_h) &=0,\quad \forall  q_h\in \Lambda_{h}^{k+1}.
    \end{aligned}
\end{equation}

\subsection{Approximation and convergence analysis}
In this section, we establish the approximation and convergence properties of the finite element discretization
for the Hodge Laplacian advection-diffusion eigenvalue problem.
In contrast to the half Hodge Laplacian advection-diffusion eigenvalue problem~\eqref{eq:gen:mixed}, the discrete compactness property (\Cref{lem:str_dis_com}) cannot be applied directly here, since the eigenfunctions $v$ of~\eqref{eq:Hodge:mixed} need not be $d$-free. Instead, we shall prove below that this mixed formulation \eqref{eq:Hodge:mixed} satisfies a discrete inf--sup condition, which plays a central role in controlling the approximation error.

Adding the shift term $\nu(v,w)$ with $\nu > 0$ (to be chosen appropriately below),
we define the following sesquilinear form
$
A \colon
X
\times
X
\to \mathbb{C}
$
by
\begin{align*}
    A(( v, p),( w, q)) &= \varepsilon (\delta^k v,\delta^k w) + \varepsilon(\delta^{k+1} p, w)+ (\mathrm{i}_u^{k+1} p, w) + (\mathrm{i}_u^k v,\delta^k w)\\
    & \ \ \ \ \ \ +( p, q) - ( v,\delta^{k+1} q) +\nu ( v, w),
\end{align*}
where
$$
X:= \mathring{H}^*\Lambda^k(\Omega) \times \mathring{H}^*\Lambda^{k+1}(\Omega).
$$

 Below we use the Babu\v{s}ka-Brezzi theory to establish the well-posedness of the source problem and define the solution operator, which is a necessary ingredient of the convergence analysis of the discrete eigenvalue problem. The boundedness of the sesquilinear forms is straightforward under boundedness assumptions on $u$. Thus, the well-posedness of the solution operator boils down to the following inf-sup condition.
\begin{lemma}[Inf-sup conditions]\label{lem:infsup}
For any sufficiently large shift
$\nu \geq 4\varepsilon + \frac{1}{2}\|u\|_{L^\infty(\Omega)} + \varepsilon^{-1}{\|u\|_{L^\infty(\Omega)}^2} + 1 > 0$,
there exist constants
$\alpha = \min\bigl\{\frac{\varepsilon}{4},\,\frac{1}{2}\bigr\} > 0$
and $C > 0$ such that,
for every $(v,p) \in X$,
there exists $(w,q) \in X$ satisfying
    \begin{align*}
        |A(( v, p),( w, q))| \geq \alpha (\| v\|_{H^*\Lambda} +\| p\|_{H^*\Lambda})^2,
    \end{align*}
    and
    \begin{equation*}
        (\| w\|_{H^*\Lambda} +\| q\|_{H^*\Lambda})\leq C (\| v\|_{H^*\Lambda} +\| p\|_{H^*\Lambda});
    \end{equation*}
    and for any $( w, q) \in X$ with $(w,q)\neq 0$, there exists $( v, p) \in X$ such that
        \begin{align*}
        |A(( v, p),( w, q))| > 0.
    \end{align*}
\end{lemma}
\begin{proof}
    For any $( v, p) \in X$, choose $( w, q) = ( v+\delta^{k+1} p,(\varepsilon+\nu) p)$. Then we have
    \begin{align*}
        A(( v, p),( w, q)) & = \varepsilon(\delta^k v,\delta^k v) + \varepsilon(\delta^{k+1} p,\delta^{k+1} p) + \varepsilon ((\delta^{k+1} p, v) - ( v,\delta^{k+1} p))\\ & \ \ \ \ \ \ \ + (\mathrm{i}_u^{k+1} p, v) + (\mathrm{i}_u^{k+1}p,\delta^{k+1} p)+ (\mathrm{i}_u^k v,\delta^k v)   + (\varepsilon+\nu)( p, p) + \nu ( v, v).
        \end{align*}
Under the assumption $u\in L^{\infty}$, we get
        \begin{equation*}
            |\varepsilon ((\delta^{k+1} p, v) - ( v,\delta^{k+1} p))|\leq \frac{\varepsilon}{4}\|\delta^{k+1} p\|_{L^2}^2 + 4\varepsilon\| v\|_{L^2}^2,
        \end{equation*}
        \begin{equation*}
            |(\mathrm{i}_u^{k+1} p, v)|\leq \frac{1}{2}\|u\|_{L^{\infty}}(\| p\|_{L^2}^2 + \| v\|_{L^2}^2),
        \end{equation*}
        \begin{equation*}
            |(\mathrm{i}_u^{k+1} p,\delta^{k+1} p)+ (\mathrm{i}_u^k v,\delta^k v)|\leq \frac{\varepsilon}{4}(\|\delta^{k+1} p\|_{L^2}^2 + \|\delta^{k} v\|_{L^2}^2) + \frac{\|u\|_{L^{\infty}}^2}{\varepsilon}(\| p\|_{L^2}^2 + \| v\|_{L^2}^2).
        \end{equation*}
Choose $\nu \geq 4\varepsilon + \frac{1}{2}\|u\|_{L^{\infty}} + \varepsilon^{-1}{\|u\|_{L^{\infty}}^2} + 1>0$. There exist positive constants $\alpha>0$ and $C>0$ such that
        \begin{align*}
            |A(( v, p),( w, q))|&\geq |\varepsilon(\delta^k v,\delta^k v) + \varepsilon(\delta^{k+1} p,\delta^{k+1} p)+ (\varepsilon+\nu)( p, p) + \nu ( v, v)|\\
            & \ \ \ \ \ \ \ - | \varepsilon ((\delta^{k+1} p, v) - ( v,\delta^{k+1} p)) + (\mathrm{i}_u^{k+1} p, v) + (\mathrm{i}_u^{k+1} p,\delta^{k+1} p)+ (\mathrm{i}_u^k v,\delta^k v)|\\
            &\geq \alpha (\| v\|_{H^*\Lambda} +\| p\|_{H^*\Lambda})^2,
        \end{align*}
and
    \begin{equation*}
        \| w\|_{H^*\Lambda} +\| q\|_{H^*\Lambda}\leq C (\| v\|_{H^*\Lambda} +\| p\|_{H^*\Lambda}).
    \end{equation*}
    Next, for any $( w, q) \in X$, if $w-{\nu^{-1}}\delta^{k+1}q\neq 0$, choose $(v,p) = (w-\nu^{-1}\delta^{k+1}q,0)$ such that
    \begin{align*}
        |A(( v, p),( w, q))| = |\varepsilon(\delta^kw,\delta^kw) +(\mathrm{i}_u^kv,\delta^kw) + \nu(v,v)|\geq \alpha (\|\delta^kw\|_{L^2}+ \|w-\frac{1}{\nu}\delta^{k+1}q\|_{L^2})^2>0.
    \end{align*}
    Otherwise, choose $(v,p) = (0,\nu q)$ such that
    \begin{align*}
        |A(( v, p),( w, q))| = |\varepsilon(\delta^{k+1}q,\delta^{k+1}q) + (\mathrm{i}_u^{k+1}q,\delta^{k+1}q)+\nu(q,q)|\geq\alpha\|q\|_{H^*\Lambda}^2> 0,
    \end{align*}
    which completes the proof.
\end{proof}
By the inf-sup condition, we can define the solution operator $T:L^2\Lambda^k(\Omega)\rightarrow \mathring{H}^*\Lambda^{k}(\Omega)$ as
\begin{subequations}
\begin{align}
    \varepsilon (\delta^k T f,\delta^k w) + \varepsilon(\delta^{k+1} p, w)+ (\mathrm{i}_u^{k+1} p, w) + (\mathrm{i}_u^kT f,\delta^k w) + \nu (T f, w)&= ( f, w),\quad \forall  w\in \mathring{H}^*\Lambda^{k}(\Omega),\\\label{eqn:p}
    ( p, q) - (T f,\delta^{k+1} q) &=0,\quad \forall  q\in \mathring{H}^*\Lambda^{k+1}(\Omega),
\end{align}
\end{subequations}
with some $ p\in \mathring{H}^*\Lambda^{k+1}(\Omega)$.
It follows from \eqref{eqn:p} that $p=d^kTf\in L^2\Lambda^{k+1}(\Omega)$. Consequently, $Tf\in H\Lambda^k(\Omega)\cap \mathring{H}^*\Lambda^{k}(\Omega) $ and satisfies, for any $w\in  H\Lambda^k(\Omega)\cap \mathring{H}^*\Lambda^{k}(\Omega)$,
\begin{equation*}
    \varepsilon (\delta^kTf,\delta^k w) + \varepsilon (d^k  Tf,d^k w) + (\mathrm{i}_u^{k+1}d^k Tf, w) + (\mathrm{i}_u^{k} Tf,\delta^k w) +\nu (Tf,w) = ( f, w).
\end{equation*}
Hence, $T:L^2\Lambda^k(\Omega)\rightarrow L^2\Lambda^k(\Omega)$ is compact. Similarly, we can define the discrete solution operator  $T_h:L^2\Lambda^k(\Omega)\rightarrow \Lambda^k_h$ as
\begin{equation}\label{eq:fullhodge_dis_mix}
    \begin{aligned}
    \varepsilon (\delta^k T_h f,\delta^k w) + \varepsilon(\delta^{k+1} p_h, w)+ (\mathrm{i}_u^{k+1} p_h, w) + (\mathrm{i}_u^kT_h f,\delta^k w) + \nu (T_h f, w)&= ( f, w),\quad \forall  w\in \Lambda^k_h,\\
    ( p_h, q) - (T_h f,\delta^{k+1} q) &=0,\quad \forall  q\in\Lambda_h^{k+1},
\end{aligned}
\end{equation}
with some $ p_h\in \Lambda^{k+1}_h(\Omega)$. Define the discrete differential operator $d_h^k:L^{2}\Lambda^k(\Omega)\rightarrow \Lambda_h^{k+1}$ by
\begin{equation*}
    (d_h^k v,q_{h}) = (v,\delta^{k+1}q_{h}),\quad \forall q_{h}\in\Lambda_h^{k+1}.
\end{equation*}
The second equation in \eqref{eq:fullhodge_dis_mix} implies $p_h = d_h^k T_hf$. The following discrete compactness result extends the discrete compactness result in \cite{2019HeHuXu} to domains with nontrivial topology.
\begin{lemma}\label{lem:dis_compact}
    For any sequence $\{u_{h_n}\in\Lambda_{h_n}^k\}$ with $h_n\rightarrow 0$ satisfying $\|u_{h_n}\|_{L^2}+\|\delta^ku_{h_n}\|_{L^2}+\|d_{h_n}^ku_{h_n}\|_{L^2}\leq C$, where $C$ is a positive constant, there exists a subsequence of $\{u_{h_n}\}$ which converges strongly in $L^2\Lambda^k(\Omega)$.
\end{lemma}
\begin{proof}
    There exists an $L^2$-orthogonal decomposition $u_{h_n} = \tilde{u}_{h_n} + \delta^{k+1}q_{h_n}$ with $\tilde{u}_{h_n}\in \mathbb{K}^k_{h_n}$ and $q_{h_n}\in \mathbb{K}_{h_n}^{k+1}\cap (\mathcal{H}_{h_n}^{k+1})^{\perp}$. Moreover, $\|\tilde{u}_{h_n}\|_{H^*\Lambda}\leq \|u_{h_n}\|_{H^*\Lambda}$. The definition \eqref{def:Kh} implies that any element $ \tilde{u}_{h_n}\in  \mathbb{K}^k_{h_n}$ satisfies $d_{h_n}^{k} \tilde{u}_{h_n}=0$. This further implies
        \begin{equation}\label{bound:qh}
    \|\delta^{k+1}q_{h_n}\|_{L^2}+\|d_{h_n}^k\delta^{k+1}q_{h_n}\|_{L^2}\leq \|u_{h_n}\|_{L^2}+\|d_{h_n}^ku_{h_n}\|_{L^2}.
        \end{equation}
    From Lemma \ref{lem:str_dis_com}, there exists a subsequence (still denoted by $\{\tilde{u}_{h_n}\}$)
    \begin{equation*}
        \tilde{u}_{h_n}\rightarrow \tilde{u}\quad\text{strongly in }L^2.
    \end{equation*}
    From Lemma \ref{lem:str_dis_com}, the bound \eqref{bound:qh} and the discrete Poincar\'{e} inequality (cf. \cite[Theorem~5.11]{2006ArnoldFalkWinther}), there exists a subsequence, still denoted by $\{q_{h_n}\}$, such that
    \begin{equation*}
        q_{h_n}\rightarrow q\quad\text{strongly in }L^2,\qquad  q_{h_n}\rightharpoonup q\quad\text{weakly in }H^*\Lambda,\qquad d_{h_n}^{k}\delta^{k+1}q_{h_n}\rightharpoonup g\quad\text{weakly in }L^2.
    \end{equation*}
    For any smooth $(k+1)$-form $w\in C_0^{\infty}\Lambda^{k+1}(\Omega)$, note that
    \begin{equation*}
        \pi_{h_n}^{k+1}w\rightarrow w\quad\text{strongly in }L^2,\qquad \pi_{h_n}^{k}\delta^{k+1}w\rightarrow \delta^{k+1}w\quad\text{strongly in }L^2.
    \end{equation*}
Then
    \begin{align*}
        (g,w) =\lim_{n\rightarrow\infty}(d_{h_n}^{k}\delta^{k+1}q_{h_n},\pi_{h_n}^{k+1}w)
         = \lim_{n\rightarrow\infty}(\delta^{k+1}q_{h_n},\pi_{h_n}^{k}\delta^{k+1}w)
        = (\delta^{k+1}q,\delta^{k+1}w).
    \end{align*}
    Hence $g = d^k\delta^{k+1}q$, and using the vanishing boundary conditions of $q$,
    \begin{align*}
        (\delta^{k+1}q,\delta^{k+1}q) =(d^{k}\delta^{k+1}q,q) &= \lim_{n\rightarrow\infty}(d_{h_n}^{k}\delta^{k+1}q_{h_n},q_{h_n}) = \lim_{n\rightarrow\infty}(\delta^{k+1}q_{h_n},\delta^{k+1}q_{h_n}).
    \end{align*}
The above norm convergence together with the weak convergence
    \begin{equation*}
        \delta^{k+1}q_{h_n}\rightharpoonup \delta^{k+1}q\quad\text{weakly in }L^2,
    \end{equation*}
implies the strong convergence
      \begin{equation*}
        \delta^{k+1}q_{h_n}\rightarrow\delta^{k+1}q\quad\text{strongly in }L^2,
    \end{equation*}
    which completes the proof.
\end{proof}

\begin{theorem}
     Assume that $u \in L^{\infty}$ and $\nu\geq 4\varepsilon + \frac{1}{2}\|u\|_{L^{\infty}} + \varepsilon^{-1}{\|u\|_{L^{\infty}}^2} + 1>0$. We have
    \begin{equation*}
        \|T-T_h\|_{\mathcal{L}(L^2\Lambda^k(\Omega);L^2\Lambda^k(\Omega))} \rightarrow 0,\quad \text{ as }h\rightarrow 0.
    \end{equation*}
\end{theorem}
\begin{proof}
Suppose, for contradiction, that there exists $\eta>0$, $f_j\in L^2\Lambda^k(\Omega)$ with $\|f_j\|_{L^2}=1$ and $h_j\rightarrow 0$ such that
\begin{equation}\label{contradiction}
    \|(T-T_{h_j})f_j\|_{L^2}\geq \eta,\quad \text{for }j=1,2,\cdots .
\end{equation}
We can extract a subsequence, still denoted by $f_j$, such that
\begin{equation*}
    f_j\rightharpoonup f\quad \text{weakly in }L^2.
\end{equation*}
Since $T$ is a compact operator, $Tf_j\rightarrow Tf$ strongly in $L^2$. Let $p_j\in \Lambda_{h_j}^{k+1}$, $v_j = T_{h_j}f_j\in \Lambda_{h_j}^k$ be the solution of \eqref{eq:fullhodge_dis_mix}. From Lemma \ref{lem:infsup}, we have
\begin{equation*}
    \|p_j\|_{H^*\Lambda} +\|v_j\|_{H^*\Lambda}\leq C\|f_j\|_{L^2}.
\end{equation*}
Note that
\begin{equation*}
    d_{h_j}^{k}v_j = p_j,\qquad d_{h_j}^{k+1}p_j =0.
\end{equation*}
From Lemma \ref{lem:dis_compact}, we can extract subsequences, still denoted by $\{v_j\}$ and $\{p_j\}$, such that
\begin{equation*}
    p_j\rightarrow p\quad\text{strongly in }L^2, \qquad p_j\rightharpoonup p\quad\text{weakly in }H^*\Lambda,
\end{equation*}
\begin{equation*}
   v_j\rightarrow v\quad\text{strongly in }L^2,\qquad v_j\rightharpoonup v\quad\text{weakly in }H^*\Lambda,
\end{equation*}
for some $p$ and $v$.
For any $w\in C_0^{\infty}\Lambda^k(\Omega)$, note that
\begin{equation*}
    w_j=\pi_{h_j}^kw\rightarrow w\quad\text{strongly in }H^*\Lambda.
\end{equation*}
Then
\begin{align*}
    &\varepsilon (\delta^k v,\delta^k w) + \varepsilon(\delta^{k+1}p, w)+ (\mathrm{i}_u^{k+1} p, w) + (\mathrm{i}_u^kv,\delta^k w) + \nu (v, w)\\
    &=\lim_{j\rightarrow\infty}\varepsilon (\delta^k v_j,\delta^k w_j) + \varepsilon(\delta^{k+1} p_j, w_j)+ (\mathrm{i}_u^{k+1} p_j, w_j) + (\mathrm{i}_u^kv_j,\delta^k w_j) + \nu (v_j, w_j)\\
    & = \lim_{j\rightarrow \infty}(f_j,w_j) = (f,w).
\end{align*}
Similarly, for any $q\in C_0^{\infty}\Lambda^{k+1}(\Omega)$, we obtain
\begin{equation*}
    (p,q)-(v,\delta^{k+1}q)=0.
\end{equation*}
Hence $v = Tf$. Since $Tf_{j}\rightarrow Tf=v$ strongly in $L^2$, $T_{h_j}f_j=v_{j}$, and $v_{j}\rightarrow v$ strongly in $L^2$, this contradicts the assumption \eqref{contradiction}.
\end{proof}

We conclude this section by introducing the error estimate of the eigenvalue problem (cf. \cite[Section~9]{2010Boffi}). Recall that $X: =\mathring{H}^*\Lambda^{k}(\Omega)\times \mathring{H}^*\Lambda^{k+1}(\Omega)$ and define
\begin{equation*}
    B((u,p),(w,q)) = (u,w),\quad \forall (u,p),(w,q)\in X.
\end{equation*}
From the inf-sup conditions in Lemma \ref{lem:infsup}, we can define $\tilde{T}: X\rightarrow X$, $\tilde{T}^*:X\rightarrow X$ by
\begin{equation*}
    A(\tilde{T}(u,p),(w,q)) = B((u,p),(w,q)),\quad \forall (w,q)\in X,
\end{equation*}
\begin{equation*}
    A((u,p),\tilde{T}^*(w,q)) = B((u,p),(w,q)),\quad \forall (u,p)\in X.
\end{equation*}
Similarly for any $(u,p)\in X_h:=\Lambda_h^k\times\Lambda_h^{k+1}$, we can define $\tilde T_h:X_h\to X_h$ by
\begin{equation*}
    A(\tilde{T}_h(u,p),(w,q)) = B((u,p),(w,q)),\quad \forall (w,q)\in X_h.
\end{equation*}
It follows that $T$ and $\tilde{T}$ have the same spectrum, and $T_h$ and $\tilde{T}_h$ have the same nonzero spectrum, with the same algebraic multiplicities. Let $\lambda$ be an eigenvalue of \eqref{eq:Hodge:mixed} with algebraic multiplicity $m$ and ascent $\alpha$. Then $\overline{(\lambda+\nu)^{-1}}$ will be an eigenvalue of $\tilde T^*$ with algebraic multiplicity $m$. Let $E$ be the space of generalized eigenvectors of $\tilde{T}$ corresponding to $(\lambda+\nu)^{-1}$ and $E^*$ be the space of generalized eigenvectors of $\tilde T^*$ corresponding to $\overline{(\lambda+\nu)^{-1}}$. Then there exist $m$ eigenvalues of \eqref{eq:Hodge:mixed:dis}, denoted by $\lambda_{1h},\cdots,\lambda_{mh}$, converging to $\lambda$, and the following holds (\cite[Corollaries~9.6, 9.8]{2010Boffi}).

\begin{theorem}
    There exist constants $C>0$ and $h_0>0$ such that, for $0<h<h_0$,
    \begin{equation*}
        |\lambda-\frac{1}{m}\sum_{i=1}^m\lambda_{i,h}|\leq C\delta_h\delta_h^*
    \end{equation*}
    \begin{equation*}
        |\lambda-\lambda_{i,h}|\leq C(\delta_h\delta_h^*)^{1/\alpha},\quad i=1,2,\cdots,m,
    \end{equation*}
    with
    \begin{equation*}
        \delta_h = \sup_{\substack{U\in E\\ \|U\|_{V}=1}}\inf_{U_h\in X_h}\|U-U_h\|_X,
    \end{equation*}
    and
    \begin{equation*}
        \delta_h^* = \sup_{\substack{U\in E^*\\ \|U\|_{V}=1}}\inf_{U_h\in X_h}\|U-U_h\|_X,
    \end{equation*}
    where, for any $U = (v,q)\in X$, we set $\|U\|_{X} = \|v\|_{H^*\Lambda}+\|q\|_{H^*\Lambda}$ and $\|U\|_{V} = \|v\|_{L^2}$.
\end{theorem}
\section{Numerical experiments}\label{sec:numerical}
In this section, we describe the numerical experiments associated with the advection--diffusion eigenvalue problem~\eqref{eq:advectionDiffusionOperator} across the dual de~Rham complex.
The experiments are organized according to the formulations they test: the full Hodge mixed scheme~\eqref{eq:Hodge:mixed} of Section~\ref{sec:convergence2} for the bottom-of-the-complex degree $\Lambda^0$, and the degenerate top-form case ($\Lambda^3$ in 3D) in which the codifferential of the multiplier vanishes identically and the half-Hodge formulation~\eqref{eq:gen:mixed} of Section~\ref{sec:convergence1} reduces to the primal Galerkin advection--diffusion eigenproblem.
The half-Hodge mixed formulation~\eqref{eq:gen:mixed} of Section~\ref{sec:convergence1} and its potential-based alternative of Remark~\ref{rmk:alternative-formulation} are compared directly on a re-entrant Fichera cube in Section~\ref{sec:numerical:lshape}.
The intermediate degrees (the Maxwell / kinematic-dynamo and magnetic-field cases) are treated in the manifold experiments of Section~\ref{sec:numerical:surface}.
All the numerical experiments performed in this section were implemented using Firedrake and Netgen \cite{firedrake, netgen, ngsPETSc}, and are available at \url{https://github.com/UZerbinati/MHD}.

\subsection{Dynamo theorems}\label{sec:numerical:theorems}
The numerical experiments described below are designed to validate, where possible, the following theorems concerning stationary forms of the magnetic advection--diffusion operator $\mathcal{L}_{\vec{u}}$~\cite{arnoldkhesin2021}.
Throughout this section $\mathcal{L}_{\vec{u}}$ denotes the coercive advection--diffusion operator $\varepsilon\, d^{k-1}\delta^{k} + L_{\vec{u}}$ of Section~\ref{sec:convergence1}, which is the negative of the dynamo operator~\eqref{eq:advectionDiffusionOperator} of the introduction. With this sign convention the diffusion term is positive semidefinite, so the reference spectra recorded below are non-negative and consistent with King's Theorem~\ref{thm:king}.
\begin{theorem}[V.~I.~Arnold]
    \label{thm:Arnold1}
    The number of linearly independent stationary k-forms is not less than the k-th Betti number of the manifold $\mathcal{M}$.
\end{theorem}
\begin{theorem}[V.~I.~Arnold]
    \label{thm:Arnold2}
    If the diffusion coefficient $R_m^{-1}$ is sufficiently large, then the number of linearly independent stationary k-forms is equal to the k-th Betti number of the manifold $\mathcal{M}$.
\end{theorem}
\begin{theorem}[C.~King]
    \label{thm:king}
    If the velocity field $\mathbf{u}$ is a potential field, i.e., $\mathbf{u} = \nabla \phi$ for some scalar function $\phi$, then the spectrum of the dynamo operator acting on k-forms is real and non-negative.
\end{theorem}

\subsection{Common setup}\label{sec:numerical:setup}
The Cartesian experiments are posed on the cube $[0,\pi]^3$, and the manifold experiments are posed on the flat $3$-torus and the flat $3$-cylinder described in Subsection~\ref{sec:numerical:surface}.

For every experiment, we compute the spectrum of the discrete operator at every diffusivity considered, and, at the representative Reynolds number $R_m = 1$, the corresponding pseudospectrum.
The diffusivity $R_m^{-1}$ is swept over $\{10,1,10^{-1},10^{-2}\}$ to probe the transition between diffusion-dominated and advection-dominated regimes, while the velocity field $\vec{u}$ is selected from the catalog of Table~\ref{tab:winds} below.
The same wind catalog is used across all the experimental configurations, so that a given wind label denotes the same field throughout. All velocity fields in Table~\ref{tab:winds}  belong to \(L^\infty\), as required by the analysis in Sections \ref{sec:convergence1} and \ref{sec:convergence2}. Wind~5 is the only field not in \(W^{1,\infty}\), and hence probes the \(L^\infty  \setminus W^{1,\infty}\) regime covered by the present theory but excluded by existing analyses of advection–diffusion of differential forms \cite{2013HeumannHiptmair,2016HeumannHiptmairPagliantini,heumann2012fully}.

The Cartesian discretizations used in the experiments are summarized in Table~\ref{tab:experiments}.
The form index is the index of the trial space in the dual de~Rham complex of Section~\ref{sec:notation}, consistent with the convention adopted at~\eqref{eq:deRham_dual}. The discrete spaces are the standard first-family Nédélec, Raviart--Thomas, continuous Lagrange, and discontinuous Lagrange finite element spaces, taken at the lowest polynomial order throughout. At $\vec{u} = \vec{0}$ these discretizations reproduce the analytical reference spectrum, converging to it at the expected $\mathcal{O}(h^2)$ rate under refinement to $N = 64$ (Table~\ref{tab:convergence}).

\begin{table}[h]
\centering
\footnotesize
\setlength{\tabcolsep}{4pt}
\resizebox{\textwidth}{!}{%
\begin{tabular}{cccccc}
\toprule
$d$ & form & trial $\times$ multiplier & formulation & boundary conditions & reference spectrum on $[0,\pi]^d$ \\
\midrule
3 & $\Lambda^0$ & $\mathrm{DG} \times \mathrm{RT}$ & full Hodge~\eqref{eq:Hodge:mixed} & natural & $\{l^2 + m^2 + n^2 : l,m,n \geq 1\}$ \\
3 & $\Lambda^3$ & $\mathrm{CG}$ & primal (top-form) & $v = 0$ & $\{l^2 + m^2 + n^2 : l,m,n \geq 1\}$ \\
\bottomrule
\end{tabular}}
\caption{Discretizations of the Cartesian experiments of this paper, posed on the cube $[0,\pi]^3$. The discrete spaces in the third column refer to the trial space for the eigenfunction and (when present) the auxiliary space enforcing the codifferential constraint of the formulation. The reference spectrum consists of the analytical eigenvalues used as a self-consistency check at $\vec{u} = \vec{0}$. The bottom-of-the-complex case $\Lambda^0$ exercises the full Hodge mixed formulation~\eqref{eq:Hodge:mixed}, while at the top of the complex ($\Lambda^3$ in 3D) the codifferential of the multiplier vanishes identically and the formulation collapses to the primal Galerkin advection--diffusion eigenproblem.}
\label{tab:experiments}
\end{table}

\begin{table}[h]
\centering
\footnotesize
\setlength{\tabcolsep}{4pt}
\resizebox{\textwidth}{!}{%
\begin{tabular}{clccc}
\toprule
label & velocity field $\vec{u}(\vec{x})$ in $\mathbb{R}^3$ & $\vec{u}\in W^{1,\infty}$ & $\vec{u} = \nabla \phi$ & $\nabla\cdot\vec{u} = 0$ \\
\midrule
0 & $\vec{u} = \vec{0}$ & yes & yes & yes \\
1 & constant $(1,1,1)$ & yes & yes & yes \\
2 & $(\cos y + \cos z,\, \cos z + \sin x,\, \cos x + \sin y)$ & yes & no & yes \\
3 & $(\sin z + \cos y,\, \sin x + \cos z,\, \sin y + \cos x)$ & yes & no & yes \\
4 & $\nabla\times \vec{A}$ with $\vec{A} = (\sin y \sin z,\, \sin z \sin x,\, \sin x \sin y)$ & yes & no & yes \\
5 & Heaviside shear $(\mathrm{sign}(z - \pi/2),\,0,\,0)$ & no & no & yes \\
\bottomrule
\end{tabular}}
	\caption{Wind catalog used across all the experiments. A given label denotes the same field throughout. The gradient column flags the case $\vec{u} = \nabla \phi$ in which King's Theorem~\ref{thm:king} predicts a real non-negative spectrum. The divergence-free column flags physical relevance for the dynamo problem. All the listed velocity fields belong to \(L^\infty\); wind 5 is the only one outside \(W^{1,\infty}\).}
\label{tab:winds}
\end{table}

\begin{table}[h]
\centering
\footnotesize
\begin{tabular}{ccccc}
\toprule
$N$ & $h$ & $\Lambda^0$ (DG$\times$RT) & $\Lambda^3$ (CG) & rate \\
\midrule
$8$  & $\pi/8$  & $6.3\times10^{-3}$ & $1.9\times10^{-1}$ & --- \\
$16$ & $\pi/16$ & $1.6\times10^{-3}$ & $4.8\times10^{-2}$ & $2.0$ \\
$32$ & $\pi/32$ & $4.0\times10^{-4}$ & $1.2\times10^{-2}$ & $2.0$ \\
$64$ & $\pi/64$ & $1.0\times10^{-4}$ & $3.0\times10^{-3}$ & $2.0$ \\
\bottomrule
\end{tabular}
\caption{Convergence in the self-consistency test at $\vec{u} = \vec{0}$ on the cube $[0,\pi]^3$: absolute error in the leading eigenvalue (the Dirichlet--Laplace value $3$) of the full Hodge $\Lambda^0$ and primal $\Lambda^3$ discretizations against the reference spectrum $\{l^2+m^2+n^2\}$, over an eightfold mesh refinement. The rate column is the observed order between successive listed rows; both exhibit the expected $\mathcal{O}(h^2)$ convergence rate.}
\label{tab:convergence}
\end{table}

\subsection{Full Hodge experiments}\label{sec:numerical:fullhodge}
For $k = 0$ in the dual de~Rham complex, we discretize the full Hodge mixed formulation~\eqref{eq:Hodge:mixed} of Section~\ref{sec:convergence2} using the pair $\mathrm{DG} \times \mathrm{RT}$.
With $v\in\mathrm{DG}$ and the auxiliary flux $\vec{p}\in\mathrm{RT}$ playing the role of $d^0 v$, the discrete weak form reads
\begin{equation}\label{eq:fullhodge:k0}
    \begin{alignedat}{2}
	    -\frac{1}{R_m}(\div \vec{p}, w) + (\vec{u}\cdot \vec{p}, w) &= \lambda (v, w), \quad &&\forall w\in\mathrm{DG}, \\
	    (\vec{p}, \vec{q}) + (v, \div \vec{q}) &= 0, \quad &&\forall \vec{q}\in\mathrm{RT},
    \end{alignedat}
\end{equation}
which is the dual-mixed form of $-\frac{1}{R_m}\Delta v + \vec{u}\cdot\nabla v = \lambda v$  with $\vec{p} = \nabla v$. No boundary condition is imposed on the flux space $\mathrm{RT}$, so the homogeneous Dirichlet condition $v = 0$ on $\partial\Omega$ is enforced naturally through the boundary term of the second equation, consistently with the Dirichlet reference spectrum of Table~\ref{tab:experiments}.

Stacking the unknowns $(v, \vec{p})$, the discretization~\eqref{eq:fullhodge:k0} leads to the saddle-point generalized eigenvalue problem
\begin{equation}\label{eq:generalized_eigenvalue_problem}
    \begin{bmatrix}
        0 & U - R_m^{-1} B^T \\
        B & C
    \end{bmatrix}
    \begin{bmatrix} v \\ \vec{p} \end{bmatrix}
    = \lambda
    \begin{bmatrix}
        M & 0 \\
        0 & 0
    \end{bmatrix}
    \begin{bmatrix} v \\ \vec{p} \end{bmatrix},
\end{equation}
where $M$ is the $\mathrm{DG}$ mass matrix, $C$ is the $\mathrm{RT}$ mass matrix, $B$ encodes the divergence pairing $(v, \div \vec{q})$ of the second equation, so that its transpose $B^T$ encodes the diffusion pairing $(\div \vec{p}, w)$ of the first, and $U$ represents the advection term $(\vec{u}\cdot\vec{p}, w)$.
The top-left block vanishes because \eqref{eq:fullhodge:k0} contains no pairing between $v$ and the test function $w$: both the diffusion and the advection act on $v$ only through the flux $\vec{p}$, and hence sit in the top-right block.
We write $K$ for the saddle-point matrix on the left-hand side of~\eqref{eq:generalized_eigenvalue_problem}.
Eliminating the flux through the second row, $\vec{p} = -C^{-1}B v$, reduces~\eqref{eq:generalized_eigenvalue_problem} to
\begin{equation}\label{eq:fullhodge:schur}
    \left(R_m^{-1} B^T C^{-1} B - U C^{-1} B\right) v = \lambda M v,
\end{equation}
in which $B^T C^{-1} B$ is the mixed Laplacian of the $\mathrm{DG}\times\mathrm{RT}$ pair and $U C^{-1} B$ the corresponding advection operator.
In particular, $K$ is invertible whenever the operator in~\eqref{eq:fullhodge:schur} is, and every eigenvector of~\eqref{eq:generalized_eigenvalue_problem} with a finite eigenvalue has a non-vanishing scalar component $v$.
The right-hand-side mass matrix is rank deficient: its kernel consists of the vectors $(0, \vec{p})$, which are the eigenvectors of~\eqref{eq:generalized_eigenvalue_problem} associated with the infinite eigenvalue and do not correspond to eigenvalues of $\mathcal{L}_{\vec{u}}$.

Among the standard strategies for this type of generalized eigenvalue problem, we use a Krylov--Schur method as implemented in the SLEPc library \cite{slepc}, with a shift-and-invert spectral transformation targeting eigenvalues close to a shift $\sigma$. Since the mass matrix is singular, \eqref{eq:generalized_eigenvalue_problem} cannot be reduced to a standard eigenvalue problem for $M^{-1}K$, so we consider the reciprocal problem
\begin{equation}\label{eq:reciprocal_generalized_eigenvalue_problem}
    \begin{bmatrix}
        M & 0 \\
        0 & 0
	\end{bmatrix} \begin{bmatrix} v \\ \vec{p} \end{bmatrix}
    = \mu
    \begin{bmatrix}
        0 & U - R_m^{-1} B^T \\
        B & C
	\end{bmatrix} \begin{bmatrix} v \\ \vec{p} \end{bmatrix},
\end{equation}
with $\mu$ the reciprocal of $\lambda$, in which the invertible saddle-point matrix $K$ sits on the right-hand side.
Reducing~\eqref{eq:reciprocal_generalized_eigenvalue_problem} to standard form amounts to solving
\begin{equation}\label{eq:inverted_generalized_eigenvalue_problem}
    \begin{bmatrix}
    0 & U - R_m^{-1} B^T \\
        B & C
    \end{bmatrix}^{-1}
    \begin{bmatrix}
        M & 0 \\
        0 & 0
	\end{bmatrix} \begin{bmatrix} v \\ \vec{p} \end{bmatrix}
	= \mu \begin{bmatrix} v \\ \vec{p} \end{bmatrix},
\end{equation}
which is exactly the shift-and-invert transformation of~\eqref{eq:generalized_eigenvalue_problem} with the shift $\sigma$ placed at the origin, and it is to the operator in~\eqref{eq:inverted_generalized_eigenvalue_problem} that the Krylov--Schur iteration is applied.
Each application of this operator requires one solve with the saddle-point matrix $K$, carried out either by a direct $LU$ factorization or by the preconditioned iterative solver of Section~\ref{sec:numerical:riesz}.
The kernel vectors $(0, \vec{p})$ of the mass matrix are mapped to $\mu = 0$ by~\eqref{eq:inverted_generalized_eigenvalue_problem}, that is, to the infinite eigenvalue of~\eqref{eq:generalized_eigenvalue_problem}.
They are filtered out automatically by the Krylov subspace generated by~\eqref{eq:inverted_generalized_eigenvalue_problem}, since the mass matrix vanishes on such vectors and the iteration targets the eigenvalues $\mu$ of largest modulus.

Given the non-normal nature of $\mathcal{L}_{\vec{u}}$, we also study the pseudospectrum~\cite{2005Trefethen}.
We follow the projection strategy of \cite{2005Trefethen} for large sparse matrices: rather than evaluating the resolvent norm of the full pencil~\eqref{eq:generalized_eigenvalue_problem}, we restrict the operator to the invariant subspace spanned by the computed eigenvectors and evaluate the resolvent norm of this restriction.
Let $\lambda_1,\dots,\lambda_n$ be the computed eigenvalues of~\eqref{eq:generalized_eigenvalue_problem}, $\Lambda = \operatorname{diag}(\lambda_1,\dots,\lambda_n)$, and let $\Phi$ be the matrix whose columns are the corresponding eigenvectors $(v_j, \vec{p}_j)$.
The pseudospectrum is measured in the $L^2$ norm of the scalar component, whose Gram matrix on the computed eigenvectors is
\begin{equation}\label{eq:pseudo:gram}
    G = \Phi^H \begin{bmatrix} M & 0 \\ 0 & 0 \end{bmatrix} \Phi .
\end{equation}
On the span of the eigenvectors the resolvent acts as $\Phi (zI - \Lambda)^{-1}\Phi^{-1}$, so that the reciprocal of its $L^2$ norm is
\begin{equation}\label{eq:pseudo:sigma_min}
	\mathcal{R}(z) = \left\| G^{1/2} (zI - \Lambda)^{-1} G^{-1/2} \right\|_2^{-1} = \sigma_{\min}\!\left( G^{1/2} (zI - \Lambda) G^{-1/2} \right),
\end{equation}
where $\sigma_{\min}$ denotes the smallest singular value.
Equivalently, $\mathcal{R}(z)$ is the smallest singular value of $zI - \widehat{K}$, where $\widehat{K} = G_V^{-1/2}\, (V^H K V)\, G_V^{-1/2}$ is the compression of the pencil onto an orthonormal basis $V$ of the same subspace and $G_V = V^H \operatorname{diag}(M, 0)\, V$ the associated Gram matrix. This is the form in which $\mathcal{R}(z)$ is evaluated for the magnetic $2$-form experiments of Section~\ref{sec:numerical:cartesian}.
When the eigenvectors are $L^2$-orthonormal, $G$ is the identity and $\mathcal{R}(z)$ reduces to the distance from $z$ to the spectrum, as expected for a normal operator. The departure of $G$ from the identity is the source of the non-normality visible in the pseudospectra below.
We sample $\mathcal{R}(z)$ on a uniform Cartesian grid over a square window of the complex plane centered on the leading eigenvalues, normalize it by its maximum over the window, and plot linearly spaced level sets of the normalized value to visualize the pseudospectrum.

\subsection{Potential formulation on the re-entrant Fichera corner}\label{sec:numerical:lshape}
The half-Hodge formulation of Section~\ref{sec:convergence1} and its potential-based alternative of Remark~\ref{rmk:alternative-formulation} are discretized together on the re-entrant Fichera corner $\Omega = [0,2]^3 \setminus [1,2]^3$, whose re-entrant edges and corner at $(1,1,1)$ are the canonical three-dimensional setting in which magnetic fields lose $H^1$ regularity.

In the classical magnetic-field formulation, we solve the discrete half-Hodge mixed system~\eqref{eq:gen:discrete_mixed}, seeking $\vec{B}\in \mathring{H}(\curl)$ together with a Lagrange multiplier $\psi\in H^1_0$ that enforces $\div\vec{B} = 0$, with the essential boundary condition $\vec{n}\times\vec{B} = 0$.
In the potential formulation of Remark~\ref{rmk:alternative-formulation} we instead seek $\vec{A}$ with $\vec{B} = \curl\vec{A}$ with the gauge $\div\vec{A} = 0$ imposed by a multiplier in $H^1/\mathbb{R}$.
Here $\vec{A}$ carries no essential boundary condition, since $\vec{n}\times\curl\vec{A} = 0$ is natural, and the shift $\nu$ is only used as a theoretical tool to ensure coercivity and thus plays no role in this section.
Because $\vec{B} = \curl\vec{A}$, the two continuous nonzero spectra coincide.
Both eigenproblems are solved by a shift-and-invert Krylov--Schur method as in Section~\ref{sec:numerical:fullhodge}.
At finite mesh size, however, the two discrete spectra and pseudospectra need not be identical. Numerically, we observe convergence between them as $h\rightarrow 0$. The most singular mode, concentrated at the re-entrant corner, converges most slowly under mesh refinement, exactly the loss-of-regularity behavior this domain is designed to exhibit.

Neither the spectrum nor the pseudospectrum is identical for the two formulations, but we observe that they converge to one another. We normalize the reciprocal $\sigma_{\min}(z - L) = 1/\|(z-L)^{-1}\|$ by its maximum over the plotted window and plot linearly spaced contour levels. The level set at value $\varepsilon$ corresponds to the boundary of the $\varepsilon$-pseudospectrum.

Figure~\ref{fig:pseudospectra:lshape} reports the two formulations at $R_m = 1$ for two winds.
The top row is the self-adjoint baseline $\vec{u} = \vec{0}$, for which the pseudospectra are indistinguishable concentric rings around the real eigenvalues ($\|G-I\|_F \approx 10^{-12}$), as expected for a normal operator.
The bottom row is the discontinuous Heaviside shear $\vec{u} = (\operatorname{sign}(z-1),0,0)$, which jumps across the plane $z = 1$ through the re-entrant corner.

\begin{figure}[!htbp]
\centering
\includegraphics[width=0.82\linewidth]{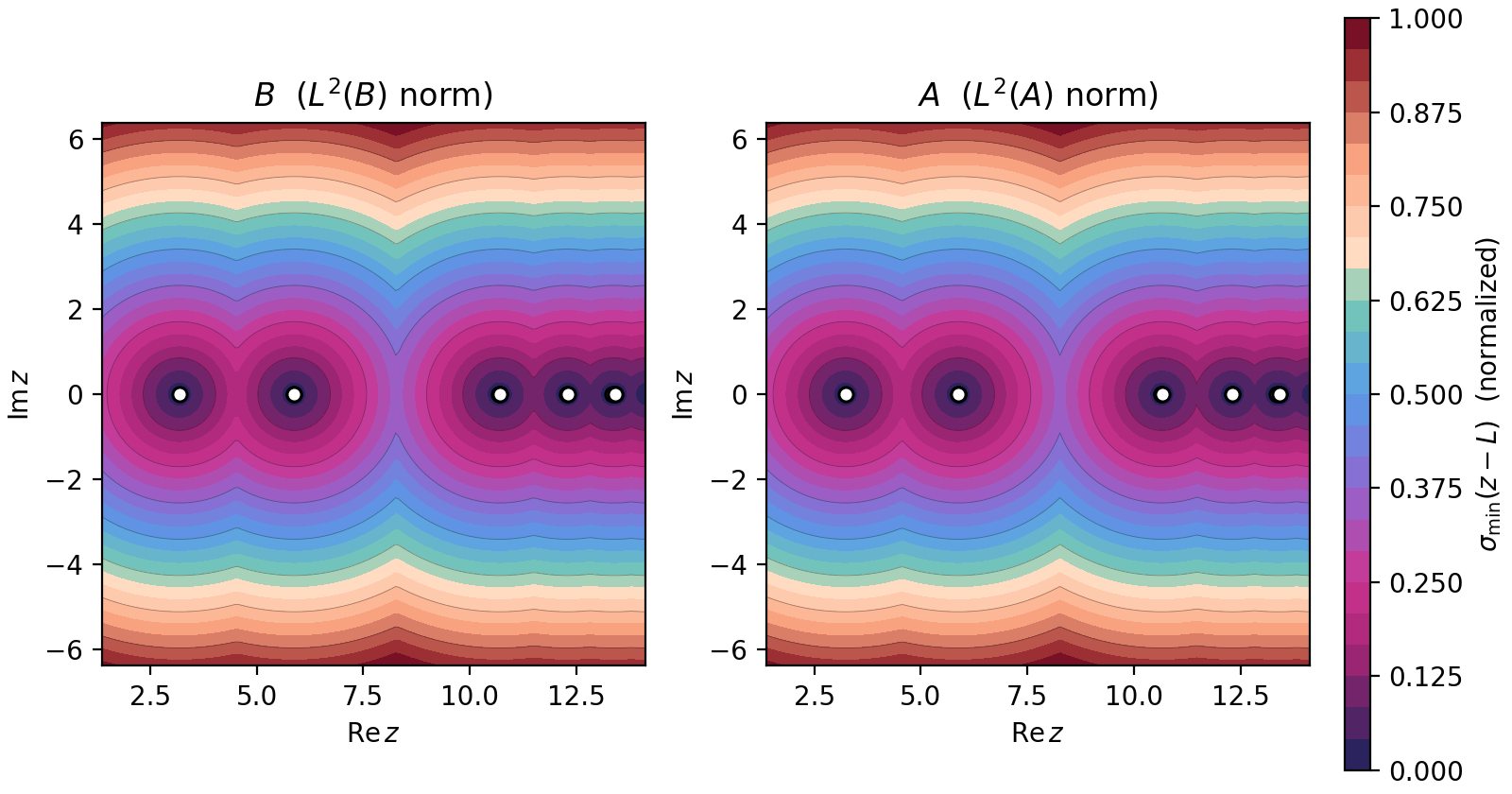} \\[3pt]
\includegraphics[width=0.82\linewidth]{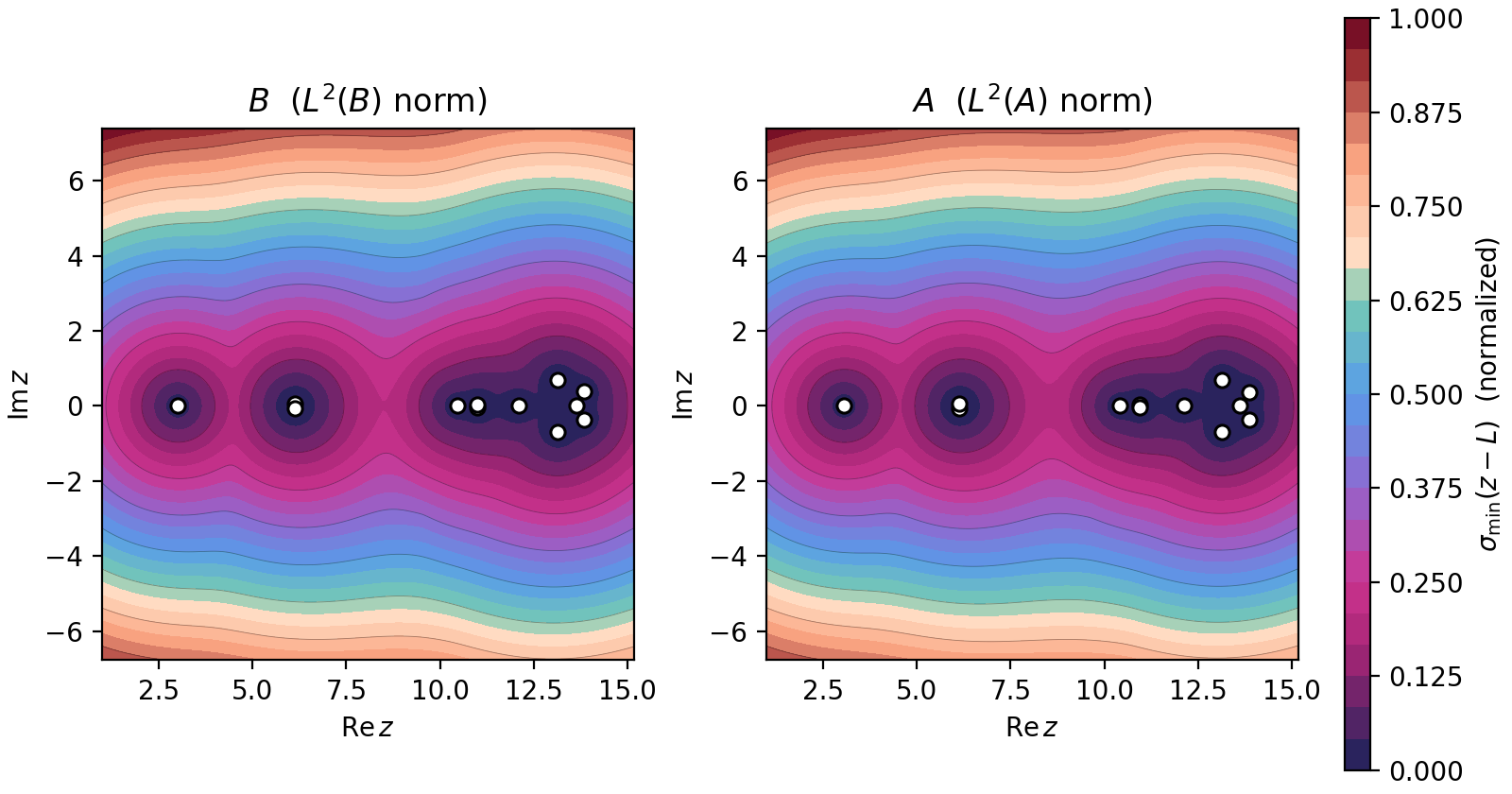}
\caption{The reciprocal of the resolvent norm, i.e., $\sigma_{\min}(z - L)$, of the dynamo problem on the re-entrant Fichera corner $\Omega = [0,2]^3\setminus[1,2]^3$ at $R_m = 1$ (Netgen mesh, $\mathrm{maxh}=0.06$ ($185{,}540$ cells)), with the computed eigenvalues overlaid as white dots. Left column of each row: classical magnetic-field formulation in the $L^2(\vec{B})$ norm. Right column: potential formulation of Remark~\ref{rmk:alternative-formulation} in the $L^2(\vec{A})$ norm. Top row: self-adjoint baseline $\vec{u} = \vec{0}$, where both operators are normal with respect to their respective $L^2$ inner products, and the pseudospectra converge to each other as $h\to 0$. Bottom row: discontinuous Heaviside shear $\vec{u} = (\operatorname{sign}(z-1),0,0)$, where the spectra and the pseudospectra still converge to each other as $h\to 0$ and both formulations exhibit comparable non-normality in their respective norms.}
\label{fig:pseudospectra:lshape}
\end{figure}

\subsection{Top-form case}\label{sec:numerical:topform}
At the top of the dual complex in three dimensions the trial space is $\Lambda^3$, realized by continuous Lagrange elements. The codifferential of the multiplier vanishes identically and the half-Hodge mixed formulation~\eqref{eq:gen:mixed} of Section~\ref{sec:convergence1} reduces to the primal Galerkin advection--diffusion eigenproblem
\begin{equation}\label{eq:fokkerplanck}
    \frac{1}{R_m}(\nabla v, \nabla w) - (v\vec{u}, \nabla w) = \lambda (v, w), \qquad \forall w\in\mathrm{CG},
\end{equation}
with $v = 0$ on $\partial\Omega$.
Eigenvalues are computed using direct shift-and-invert on the resulting Hermitian generalized eigenvalue problem when $\vec{u} = \vec{0}$ and by its non-Hermitian shift-and-invert variant otherwise. The pseudospectrum is obtained via the same reduced-model recipe~\eqref{eq:pseudo:gram}--\eqref{eq:pseudo:sigma_min} described in Section~\ref{sec:numerical:fullhodge}, with the $\mathrm{CG}$ mass matrix in place of $\operatorname{diag}(M, 0)$ in the Gram matrix~\eqref{eq:pseudo:gram}.

\subsection{Spectra and pseudospectra}\label{sec:numerical:cartesian}
Figure~\ref{fig:pseudospectra} collects the computed pseudospectra at the representative Reynolds number $R_m=1$ for the full Hodge formulation~\eqref{eq:fullhodge:k0} ($\Lambda^0$, $\mathrm{DG}\times\mathrm{RT}$, left column) and the primal top-form formulation~\eqref{eq:fokkerplanck} ($\Lambda^3$, $\mathrm{CG}$, right column) on $[0,\pi]^3$, with the computed eigenvalues overlaid as white markers (dots for the compatible $\Lambda^0$ case, crosses for the primal $\Lambda^3$ case).
The four rows correspond to the four winds from the catalog used throughout the paper: the self-adjoint baseline (wind 0), the constant gradient field (wind 1), the smooth divergence-free field (wind 4), and the discontinuous Heaviside shear (wind 5).
The discrete eigenvalues are computed at every diffusivity $R_m^{-1}\in\{10,1,10^{-1},10^{-2}\}$ considered in Section~\ref{sec:numerical:setup}. Only the $R_m = 1$ pseudospectra are plotted here for legibility, since lower diffusivity compresses the spectrum into a narrow strip while higher diffusivity disperses it past the sampled rectangle.

The first two rows provide a numerical verification of King's Theorem~\ref{thm:king}, which asserts that for a potential wind $\vec{u} = \nabla\phi$ the spectrum of the advection--diffusion operator $\mathcal{L}_{\vec{u}}$ is real and non-negative.
The self-adjoint baseline $\vec{u} = \vec{0}$ (wind 0) is the trivial gradient $\phi = 0$; the discrete spectra of both formulations sit exactly on the non-negative real axis and the pseudospectra form concentric rings around each eigenvalue, in agreement with the self-adjoint prediction.
For wind 1 the velocity field is a non-trivial gradient ($\vec{u} = (1,1,1) = \nabla(x + y + z)$): the operator is no longer self-adjoint, yet the discrete eigenvalues remain real and non-negative for both formulations, exactly as predicted by Theorem~\ref{thm:king}.
The pseudospectra at wind 1 are visibly distorted compared with the wind 0 baseline, which is consistent with the non-normality introduced by the advection term while the spectrum itself remains on the real axis.

Wind 4, a smooth divergence-free vector field that is \emph{not} a gradient, produces eigenvalues with non-vanishing imaginary parts (visible as complex-conjugate pairs of white markers off the real axis), with the pseudospectra warped accordingly.
Wind 5, the Heaviside shear in $L^\infty\setminus W^{1,\infty}$, lies outside the common regularity hypotheses for advection-diffusion of differential forms \cite{2016HeumannHiptmairPagliantini,2026LiLuoWangWu,2024WangWu,heumann2012fully,2013HeumannHiptmair,2020WuXu}. The computed eigenvalues acquire small but nonzero imaginary parts at $R_m=1$ (complex-conjugate pairs, milder than those of wind~4), but it is covered by the \(L^\infty\)-based convergence theory of Sections \ref{sec:convergence1}--\ref{sec:convergence2}.

\begin{figure}[!htbp]
\centering
\includegraphics[width=0.62\linewidth]{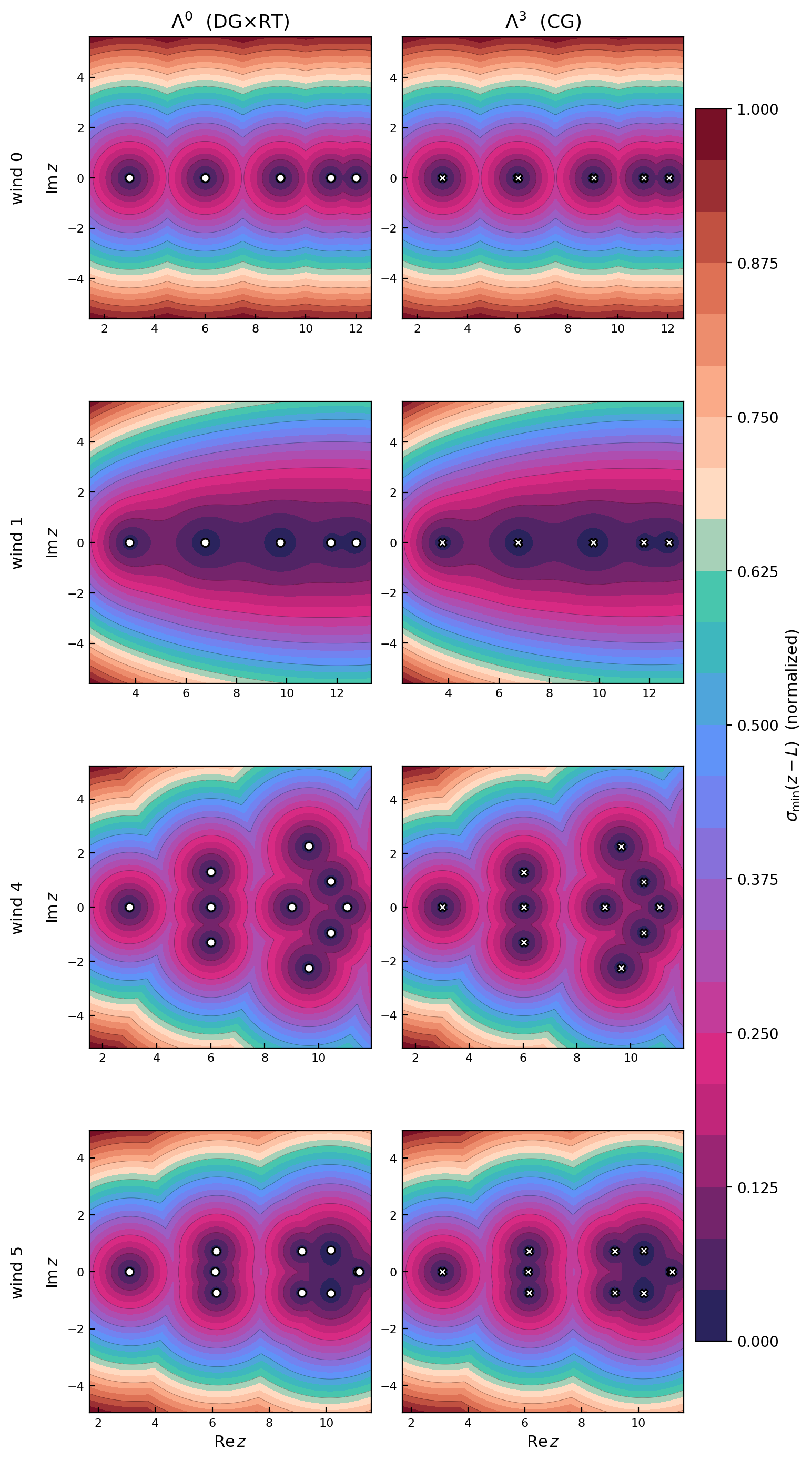}
\caption{Pseudospectra $\sigma_{\min}(z - L)$ at $R_m = 1$ on the cube $[0,\pi]^3$ ($N = 64$). Rows: wind 0 (self-adjoint), wind 1 (constant gradient), wind 4 (smooth non-gradient), wind 5 (Heaviside shear, $u\in L^\infty\setminus W^{1,\infty}$). Left column: full Hodge formulation~\eqref{eq:fullhodge:k0} ($\Lambda^0$, $\mathrm{DG}\times\mathrm{RT}$), discrete eigenvalues plotted as white dots. Right column: primal top-form formulation~\eqref{eq:fokkerplanck} ($\Lambda^3$, $\mathrm{CG}$), eigenvalues as white crosses; each panel is zoomed to the leading eigenvalues and the shared colorbar shows $\sigma_{\min}$ normalized per panel. The wind 0 and wind 1 rows numerically verify King's Theorem~\ref{thm:king} (real non-negative spectrum for potential winds $\vec{u} = \nabla\phi$); the wind 4 row exhibits complex-conjugate eigenvalues and warped pseudospectra characteristic of non-normality, and the wind 5 row shows the same effect more mildly.}
\label{fig:pseudospectra}
\end{figure}

We compute the pseudospectra of the dynamo operator for the magnetic $2$-form $\vec{B}\in H(\div)$ on the flat $3$-torus $\mathbb{T}^3$ (periodic in all three directions, $b_2 = 3$) and the flat $3$-cylinder $\mathbb{T}^2\times[0,2\pi]$ (periodic in two directions, $b_2 = 1$). The operator is the mixed $H(\div)\times H(\curl)$ realization of the $2$-form Hodge Laplacian with the induction advection term; the $3$-cylinder again uses the extruded tensor-product Raviart--Thomas and N\'ed\'elec elements.
Figure~\ref{fig:pseudospectra:kform1:flat} reports the pseudospectra $\sigma_{\min}(z - L)$ at the representative Reynolds number $R_m = 1$ for the same four winds as in Figure~\ref{fig:pseudospectra}, with the computed eigenvalues overlaid as white dots, sampled around the leading physical band of the operator. At the self-adjoint baseline (wind~0), the level sets are concentric around each eigenvalue, as expected for a normal operator; the advective winds progressively distort them and push the eigenvalues off the real axis, the smooth non-gradient wind~4 and the rough Heaviside shear wind~5 recording the largest departure from normality.

\begin{figure}[!htbp]
\centering
\includegraphics[width=0.72\linewidth]{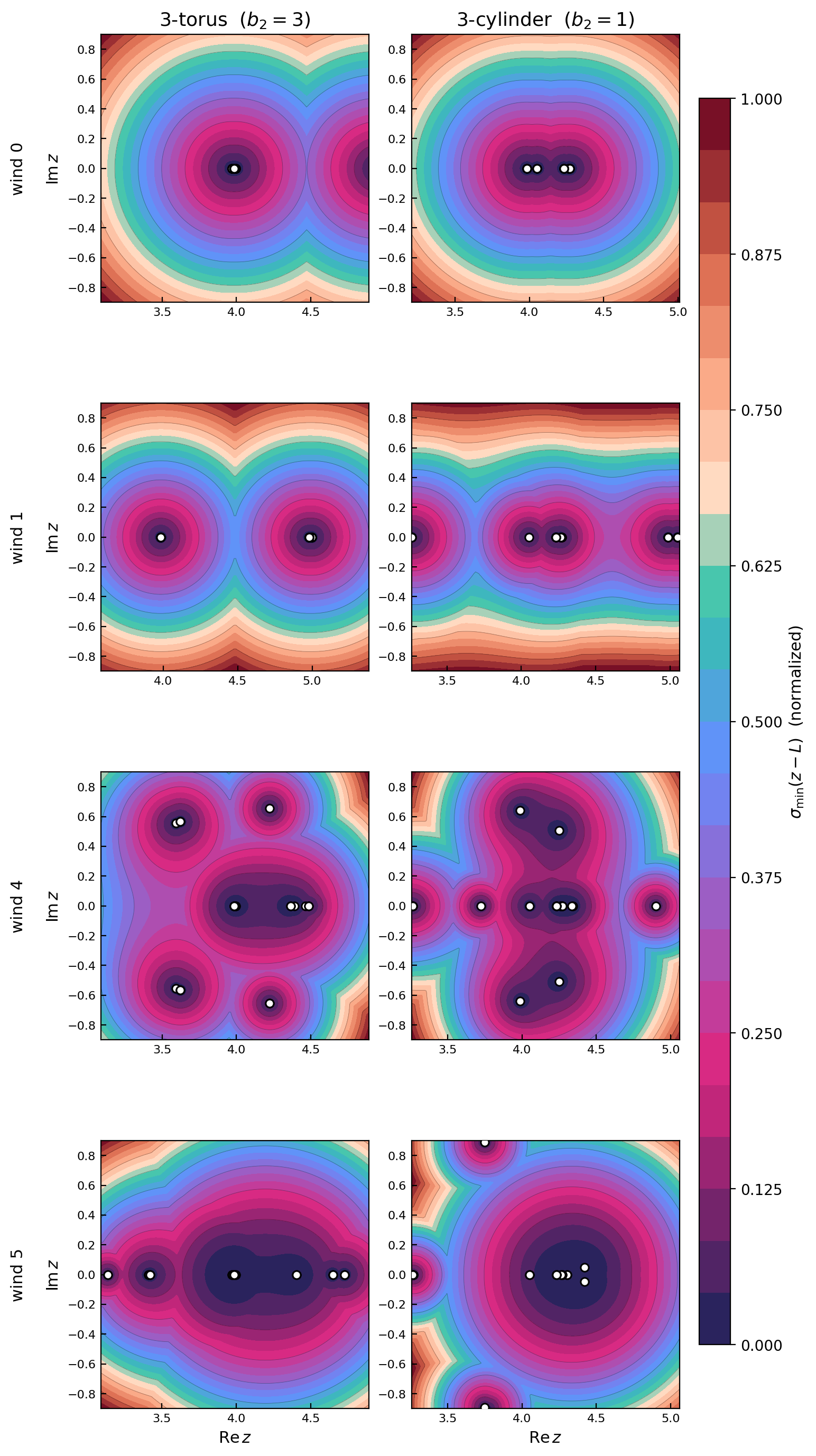}
\caption{Pseudospectra $\sigma_{\min}(z - L)$ of the magnetic $2$-form ($H(\div)$) dynamo operator at $R_m = 1$ ($N = 32$). Left column: flat $3$-torus ($b_2 = 3$). Right column: flat $3$-cylinder ($b_2 = 1$). Rows match Figure~\ref{fig:pseudospectra}: wind~0 (self-adjoint), wind~1 (constant), wind~4 (smooth non-gradient), wind~5 (Heaviside shear, $\vec{u}\in L^\infty\setminus W^{1,\infty}$). White dots are the computed eigenvalues, sampled around the leading physical band. The self-adjoint baseline gives concentric (normal) level sets on the real axis; the advective winds distort them and generate complex-conjugate pairs, with winds~4 and~5 recording the largest non-normality.}
\label{fig:pseudospectra:kform1:flat}
\end{figure}

\subsection{Iterative GMRES with Riesz-map preconditioning}\label{sec:numerical:riesz}
The shift-and-invert spectral transformation requires the action of $(A - \sigma M)^{-1}$ at every Arnoldi step, where $A$ denotes the stiffness matrix of the formulation at hand (the saddle-point matrix $K$ of~\eqref{eq:generalized_eigenvalue_problem} in the full Hodge case).
Up to moderate mesh resolutions in 2D and at coarse resolutions in 3D, this is carried out by a direct $LU$ factorization through MUMPS, but for the larger 3D experiments the fill-in of the saddle factorization becomes prohibitive.
As an alternative we use an iterative solver for the shifted linear system appearing in the Arnoldi iteration. In particular, we adopt preconditioned GMRES on the operator $A - \sigma M$, using a discrete Riesz map as the preconditioner of the relevant energy space.
Concretely, the Riesz inner product on the trial/multiplier pair from Table~\ref{tab:experiments} (for example $H(\mathrm{curl})\times H^1$ for the Kikuchi cases, $L^2 \times H(\mathrm{div})$ for the full Hodge $k=0$ case, and $H^1$ for the top-form cases) is assembled once and applied as a preconditioner $P$ for the Krylov solver.

In our experiments the action of $P^{-1}$ is computed using an $LU$ factorization of $P$ through MUMPS: since $P$ is symmetric positive definite and decouples by construction across the trial and multiplier blocks, its $LU$ factor is several times smaller than that of the indefinite saddle point system $A - \sigma M$, and remains tractable in regimes where the saddle factor does not.

\subsection{Harmonic 2-forms and Arnold's theory}\label{sec:numerical:surface}
To probe Theorems~\ref{thm:Arnold1} and~\ref{thm:Arnold2}, which relate stationary $k$-forms to the topology of the underlying manifold, we run the dynamo problem on two three-dimensional flat manifolds. By Arnold's theory, the magnetic field is an $(n-1)$-form; in three dimensions this is a $2$-form $\vec{B}\in H(\div)$ with $\div\vec{B}=0$, so the relevant Betti number is $b_2$. We use the flat $3$-torus $\mathbb{T}^3 = [0,2\pi]^3$, periodic in all three directions, with $b_2 = 3$, and the flat $3$-cylinder $\mathbb{T}^2\times[0,2\pi]$, periodic in two directions with a free interval in the third, with $b_2 = 1$.
The distinction between $b_1$ and $b_2$ is invisible on the closed $3$-torus, where Poincar\'e duality gives $b_2 = b_1 = 3$, but it is decisive on the $3$-cylinder, where a $2$-form instead counts independent \emph{fluxes}, of which there is exactly one---the flux through the $\mathbb{T}^2$ cross-section, with proxy field $e_z$ along the free interval---so $b_2 = 1$. The two \emph{circulations} $e_x, e_y$ around the periodic directions that one might expect are harmonic $1$-forms ($b_1 = 2$), not fluxes; the magnetic $2$-form does not see them.
The $3$-torus carries a simplicial (tetrahedral) mesh.

The advection is driven by the wind catalog of Table~\ref{tab:winds}, and the diffusivity is swept over $R_m^{-1}\in\{10, 1, 10^{-1}, 10^{-2}\}$. On a topologically non-trivial manifold, a constant field (wind~1) restricts to a non-exact closed 1-form, in agreement with the lower bound of Theorem~\ref{thm:Arnold1}.

To turn the verification of Theorems~\ref{thm:Arnold1} and~\ref{thm:Arnold2} into a quantitative test, we count for every $(\text{geometry}, \text{wind}, R_m)$ triple the number of computed eigenvalues with $|\lambda| \leq 10^{-8}$, filtering out the Lagrange-multiplier modes. The results are reported in Table~\ref{tab:harmonic-counts}. In every one of the 48 cases the count equals the second Betti number exactly ($3$ on the $3$-torus, $1$ on the $3$-cylinder). This confirms the lower bound of Theorem~\ref{thm:Arnold1} and attains the equality predicted by Theorem~\ref{thm:Arnold2} throughout the sweep, including at the lowest diffusivity $R_m^{-1} = 10^{-2}$ and for the smooth divergence-free non-gradient winds.

\begin{table}[h]
\centering
\footnotesize
\setlength{\tabcolsep}{6pt}
\resizebox{\textwidth}{!}{%
\begin{tabular}{c|cccc@{\hspace{1.5em}}c|cccc}
\toprule
\multicolumn{5}{c}{$3$-torus ($b_2 = 3$)} & \multicolumn{5}{c}{$3$-cylinder ($b_2 = 1$)} \\
\cmidrule(lr){1-5} \cmidrule(lr){6-10}
wind & $R_m^{-1}=10$ & $R_m^{-1}=1$ & $R_m^{-1}=10^{-1}$ & $R_m^{-1}=10^{-2}$ & wind & $R_m^{-1}=10$ & $R_m^{-1}=1$ & $R_m^{-1}=10^{-1}$ & $R_m^{-1}=10^{-2}$ \\
\midrule
0 & 3 & 3 & 3 & 3 & 0 & 1 & 1 & 1 & 1 \\
1 & 3 & 3 & 3 & 3 & 1 & 1 & 1 & 1 & 1 \\
2 & 3 & 3 & 3 & 3 & 2 & 1 & 1 & 1 & 1 \\
3 & 3 & 3 & 3 & 3 & 3 & 1 & 1 & 1 & 1 \\
4 & 3 & 3 & 3 & 3 & 4 & 1 & 1 & 1 & 1 \\
5 & 3 & 3 & 3 & 3 & 5 & 1 & 1 & 1 & 1 \\
\bottomrule
\end{tabular}}
\caption{Number of computed eigenvalues with $|\lambda|\leq 10^{-8}$ of the magnetic $2$-form ($H(\div)$) discretization on the flat $3$-torus ($b_2=3$, left) and the flat $3$-cylinder ($b_2=1$, right), as a function of the wind label (Table~\ref{tab:winds}) and the diffusivity $R_m^{-1}$. In every case the count equals $b_2$: the lower bound of Theorem~\ref{thm:Arnold1} is attained and the equality of Theorem~\ref{thm:Arnold2} holds across the whole diffusivity range, with no low-diffusion excess.}
\label{tab:harmonic-counts}
\end{table}

\subsection{ABC flow on the 3D periodic cube}\label{sec:numerical:abc}
The ABC flow on the 3D periodic cube $\mathbb{T}^3 = [0,2\pi]^3$,
\begin{equation}\label{eq:abc-wind}
    \vec{u}(x,y,z) = \bigl(\sin z + \cos y,\;\sin x + \cos z,\;\sin y + \cos x\bigr),
\end{equation}
corresponding to the parameters $A=B=C=1$, is the canonical flow used in kinematic dynamo, including that of Arnold and Korkina~\cite{arnold1983growth}.
A direct computation shows that $\nabla\times\vec{u}=\vec{u}$, and consequently the half-Hodge advection-diffusion operator $\mathcal{L}_{\vec{u}}$ acting on 2-forms admits an exact closed-form eigenpair.
Indeed, evaluating the saddle form at $(\vec{B},\psi) = (\vec{u}, 0)$ against an arbitrary test pair $(v,\phi)$ gives
\begin{align*}
    \tfrac{1}{R_m}(\nabla\times\vec{u},\,\nabla\times\vec{v})
    + (\vec{u}\times\vec{u},\,\nabla\times\vec{v})
    + (\nabla\psi,\,\vec{v}) + (\vec{u},\,\nabla\phi)
    &= \tfrac{1}{R_m}(\vec{u},\,\vec{v}) + 0 + 0 + 0 \\
    &= \tfrac{1}{R_m}\,(\vec{u},\,\vec{v}),
\end{align*}
using $\nabla\times\vec{u}=\vec{u}$ for the diffusion term, $\vec{u}\times\vec{u}=\vec{0}$ for the advection, and $\nabla\cdot\vec{u}=0$ for the constraint.  Comparing with the mass form $m((\vec{u},0),(\vec{v},\phi)) = (\vec{u},\vec{v})$ yields the eigenrelation
\begin{equation}\label{eq:abc-eigenpair}
    \mathcal{L}_{\vec{u}}\,\vec{u} \;=\; \frac{1}{R_m}\,\vec{u}.
\end{equation}
This eigenpair is exact, requires no discretization, and provides a sharp benchmark for the Kikuchi $N_1^{\mathrm{curl}}\times\mathrm{CG}$ half-Hodge formulation, i.e., we will study the rate at which the discrete eigenvalue $\lambda_h$ closest to $1/R_m$ converges to the analytic value $\lambda^* = 1/R_m$.

We discretize \eqref{eq:abc-eigenpair} on a uniform mesh of the periodic cube $[0,2\pi]^3$ with $N$ cells per direction, $N\in\{4,10,16,22,28,32,38,42\}$, using lowest-order N\'ed\'elec edge elements for $\vec{B}$ and continuous Lagrange elements for the Lagrange multiplier $\psi$.  The eigenproblem is solved with SLEPc's shift-and-invert at target $\sigma=1/R_m$. We select among the converged eigenpairs the one whose eigenvalue is closest to $\lambda^*=1/R_m$.

Figure~\ref{fig:abc} reports the empirical eigenvalue convergence (top left) alongside an oblique 3D view of the full discrete eigenfunction (top right) and four planar $z$-slices of the same eigenfunction (bottom).  The eigenvalue error decays at a rate close to $\mathcal{O}(h^2)$, the rate predicted by the Babu\v{s}ka--Osborn theory for a simple eigenvalue under the assumptions of Section~\ref{sec:convergence1}.

\begin{figure}[!htbp]
\centering
\begin{subfigure}[c]{0.48\linewidth}
\centering
\begin{tikzpicture}
\begin{loglogaxis}[
    width=\linewidth,
    height=0.85\linewidth,
    xlabel={$h = 2\pi/N$},
    ylabel={$|\lambda_h - 1/R_m|$},
    grid=both,
    grid style={dashed, gray!30},
    legend pos=south west,
    legend cell align=left,
    legend style={font=\footnotesize, fill=white, fill opacity=0.85, draw=none},
    every axis plot/.append style={line width=1.1pt, mark size=2.5pt},
    tick label style={font=\footnotesize},
    label style={font=\small},
    x dir=reverse,
]
\addplot+[teal, mark=*, mark options={solid, fill=teal!80!black}]
    table[x=h, y=lam_err, col sep=comma] {experiments/output/abc_convergence.csv};
\addlegendentry{$|\lambda_h - 1/R_m|$}
\addplot[gray, dotted, line width=0.7pt, forget plot] coordinates {(1.5708, 0.18) (0.1963, 0.00281)};
\node[gray, font=\footnotesize] at (axis cs:0.3, 0.02) {$\mathcal{O}(h^2)$};
\end{loglogaxis}
\end{tikzpicture}
\caption{Empirical eigenvalue convergence against $1/R_m$.  Dotted gray: $\mathcal{O}(h^2)$ reference.}
\label{fig:abc-convergence}
\end{subfigure}\hfill
\begin{subfigure}[c]{0.4\linewidth}
\centering
\includegraphics[width=\linewidth]{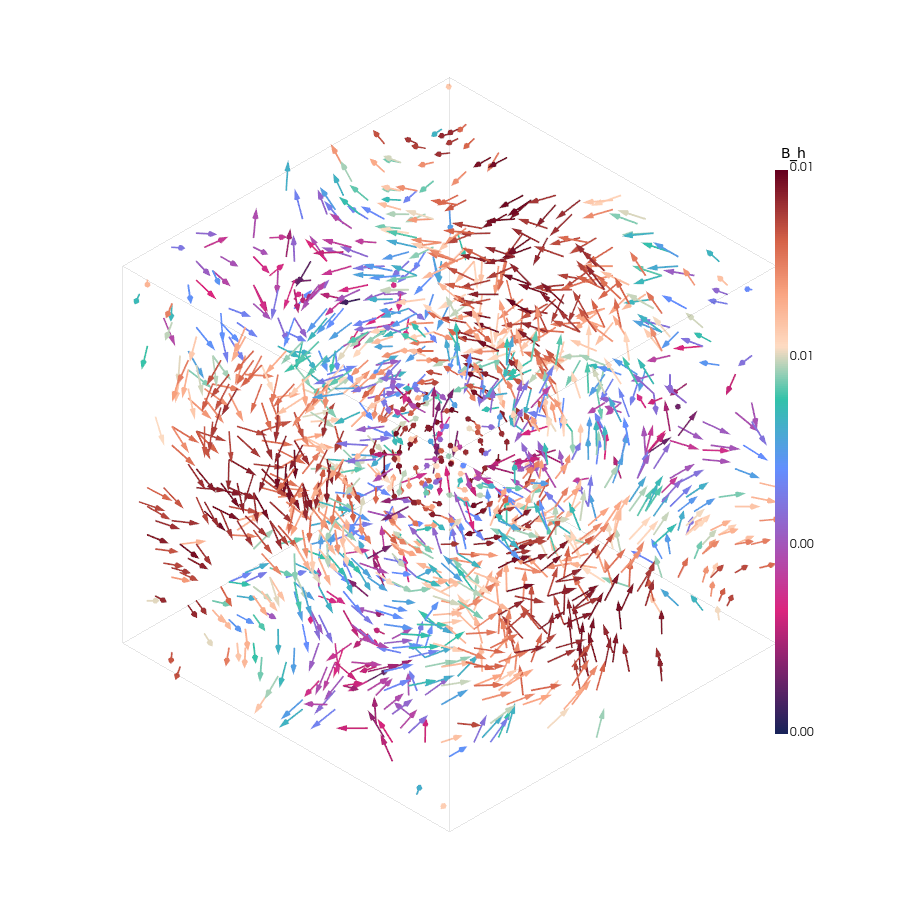}
\caption{Oblique 3D view of $\vec{B}_h$ on the finest mesh on the periodic cube $[0,2\pi]^3$.  Arrow glyphs are colored by $|\vec{B}_h|$.}
\label{fig:abc-eigenfunction-3d}
\end{subfigure}

\vspace{1ex}

\begin{subfigure}[c]{\linewidth}
\centering
\setlength{\tabcolsep}{1pt}
\renewcommand{\arraystretch}{0}
\begin{tabular}{cccc}
\includegraphics[width=0.24\linewidth]{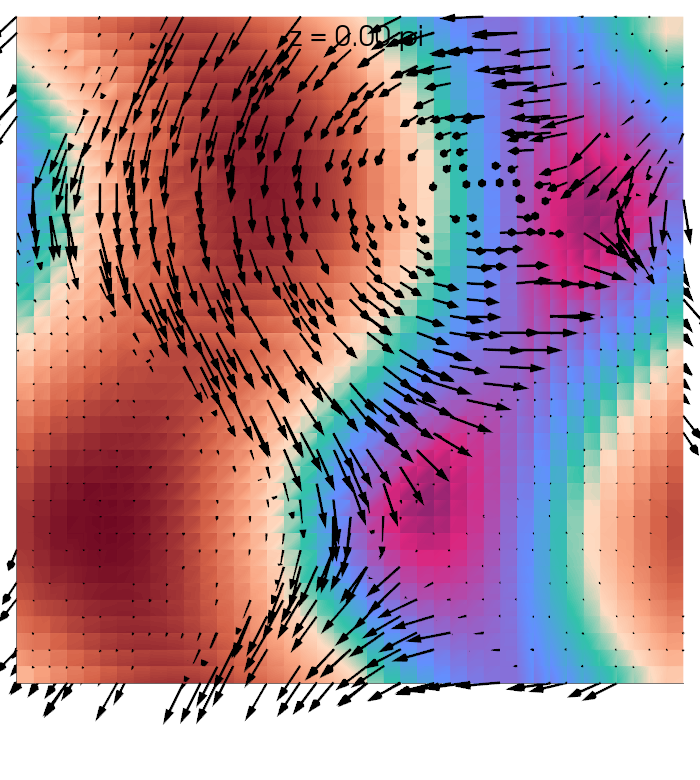} &
\includegraphics[width=0.24\linewidth]{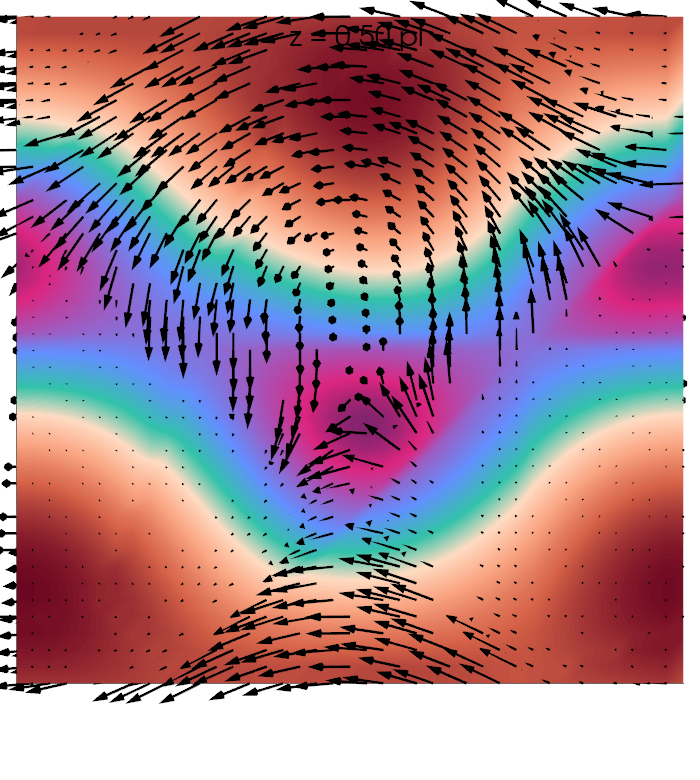} &
\includegraphics[width=0.24\linewidth]{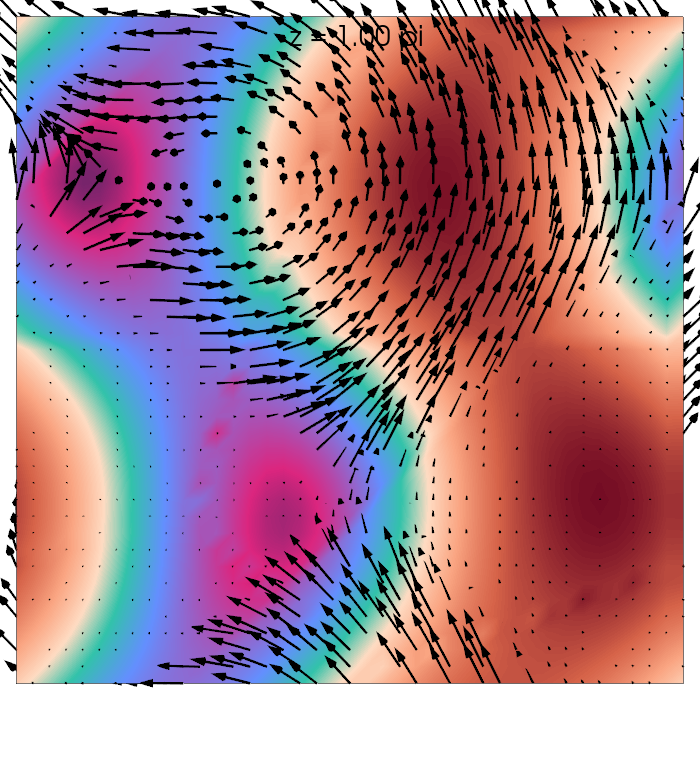} &
\includegraphics[width=0.24\linewidth]{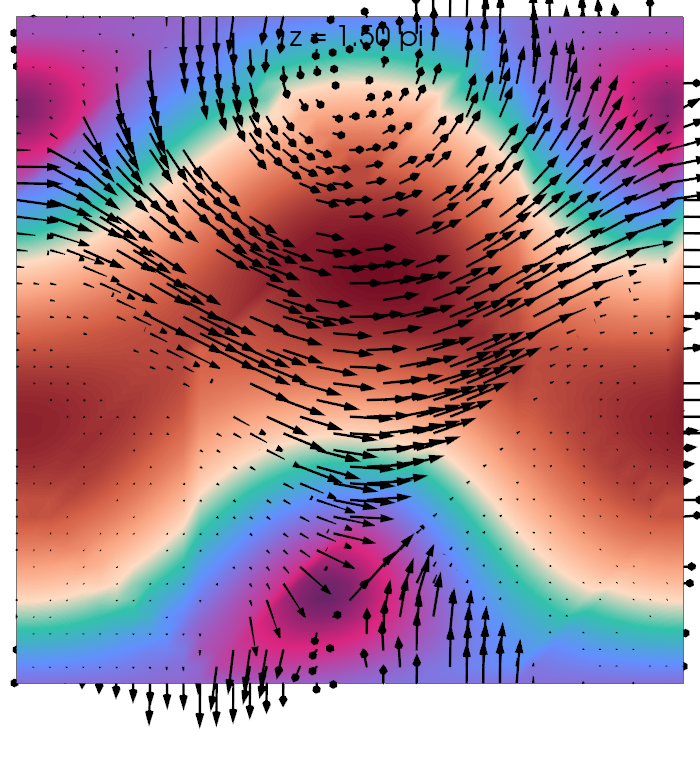} \\
\end{tabular}
\caption{Four $z$-slices of $\vec{B}_h$ at $z = 0,\,\pi/2,\,\pi,\,3\pi/2$ (left to right).  Slice color is $|\vec{B}_h|$; arrows show the projected in-plane vector field.}
\label{fig:abc-eigenfunction-slices}
\end{subfigure}
\caption{ABC flow on the 3D periodic cube at $R_m=1$.  The eigenvalue tracks the Babu\v{s}ka--Osborn $\mathcal{O}(h^2)$ prediction; the 3D view and the four $z$-slices exhibit the rotating structure that the FEEC discretization recovers.}
\label{fig:abc}
\end{figure}

\subsection{Validation of the dynamo theorems}\label{sec:numerical:validation}
We now revisit the theorems stated in Section~\ref{sec:numerical:theorems} in the light of the experiments above.

The three-dimensional manifold experiments of Section~\ref{sec:numerical:surface} confirm the lower bound of Theorem~\ref{thm:Arnold1} for the magnetic $2$-form ($k = n-1 = 2$): Table~\ref{tab:harmonic-counts} reports the number of near-zero eigenvalues of the $H(\div)$ discretization on the flat $3$-torus (where $b_2 = 3$) and on the flat $3$-cylinder (where $b_2 = 1$), and for every $(\vec{u}, R_m)$ combination the count is $\geq b_2$.

Equality with $b_2$ is observed throughout Table~\ref{tab:harmonic-counts}, for every wind and diffusivity on both manifolds, in agreement with Theorem~\ref{thm:Arnold2}.
The Cartesian pseudospectra of Figure~\ref{fig:pseudospectra} verify Theorem~\ref{thm:king} in the two gradient winds in our catalog: at wind 0 (zero field, trivially the gradient of a constant) and at wind 1 ($\vec{u}=(1,1,1)$, a constant gradient on $[0,\pi]^3$) the computed eigenvalues for both the full Hodge $\Lambda^0$ and the primal $\Lambda^3$ discretizations lie on the non-negative real axis, while for the divergence-free non-gradient winds 4 and 5 of the same figure complex-conjugate eigenvalue pairs become clearly visible.
\section{Conclusion}\label{sec:conclusions}
This paper presents a finite element exterior calculus approach for solving the spectral problem associated with the advection-diffusion operator on differential forms, with applications to the numerical dynamo problem and Fokker--Planck equations. Unlike analysis for advection-diffusion of differential forms, where it is common to assume $\bm u\in W^{1, \infty}$ \cite{2016HeumannHiptmairPagliantini,2026LiLuoWangWu,2024WangWu,heumann2012fully,2013HeumannHiptmair,2020WuXu}, we proved the convergence of the proposed scheme under {the assumption of $L^{\infty}$} velocity fields. This extension covers physically relevant cases such as shocks. This is possible since spectral problems allow an arbitrary shift, which controls the advection term in the analysis.

 These results, combined with recent advances in FEEC, provide insight into the numerical treatment of open problems in dynamo theory; however, several important further steps are required, such as the development of robust solvers.

As a further step, even pseudospectra may not capture the full picture of dynamo action. Ultimately, one must return to the fully coupled nonlinear MHD system and study its long-term evolution. Further discussions are left for future investigation.

\section*{Acknowledgments}

DB is a member of the INdAM research group GNCS and his work was
partially supported by KAUST - CRG13-2025 grant (6911).
The work of KH and YL was supported by a Royal Society University Research Fellowship (URF$\backslash$R1$\backslash$221398) and the enhanced research expenses RF$\backslash$ERE$\backslash$221047. {KH was supported by an ERC Starting Grant (project 101164551, GeoFEM)}. Views and opinions expressed are however those of the authors only and do not necessarily reflect those of the European Union or the European Research Council. Neither the European Union nor the granting authority can be held responsible for them.
UZ gratefully acknowledges the Erwin Schrödinger International Institute for Mathematics and Physics (ESI) for supporting the stay in Vienna where this paper was completed and the Research Council of Norway, project number 357556 (LaVa) for supporting his research.

\section*{Use of artificial intelligence tools}
The authors used ChatGPT (OpenAI) and Claude Opus 5.5 (Anthropic) during the preparation of this manuscript. ChatGPT was used to improve the writing, identify potential errors and typographical issues in mathematical arguments and proofs, and provide suggestions on proofs and their presentation. Claude Opus 5.5 was used to provide suggestions on mathematical proofs and to improve the efficiency and style of computational code. All AI-assisted content was reviewed by the authors, and all mathematical and computational suggestions were independently verified. The authors take full responsibility for the mathematical content, numerical results, code, and final manuscript.

\bibliographystyle{siam}
\bibliography{bib}
\end{document}